\documentclass[11pt,a4paper]{amsart}
\usepackage{amsfonts}
\usepackage{amsmath}
\usepackage{amssymb}
\usepackage{mathrsfs}
\usepackage{geometry}
\usepackage{graphicx}
\usepackage{subfigure}
\usepackage{hyperref}
\usepackage{microtype}
\usepackage{enumerate}
\usepackage{epsfig}
\usepackage{cancel}
\usepackage{color}
\usepackage{cite}
\usepackage{makecell}
\usepackage{multirow}
\usepackage{footmisc}
\usepackage{wrapfig}
\usepackage{array}
\usepackage{titletoc}
\usepackage{bm}
\usepackage{times}
\hypersetup{
colorlinks=true,
linkcolor=blue,
filecolor=blue,
urlcolor=blue,
citecolor=blue,}

\newtheorem{theorem}{Theorem}[section]
\newtheorem{lemma}[theorem]{Lemma}

\newtheorem{proposition}[theorem]{Proposition}

\numberwithin{equation}{section}

\theoremstyle{remark}
\newtheorem{remark}{Remark}[section]

\allowdisplaybreaks[3]

\newcommand{\R}{{\mathbb{R}}}

\newcommand{\T}{{\mathbb{T}}}

\def\leq{\leqslant}
\def\geq{\geqslant}

\def\om{\omega}
\def\th{\theta}
\def\na{\nabla}
\def\De{\Delta}
\def\pa{\partial}
\def\pat{\partial_t}

\def\lan{\langle}
\def\ran{\rangle}

\def\re{\operatorname{Re}}
\begin{document}
\title[The optimal transition threshold of the 2D Couette flow]{The optimal transition threshold for the 2D Couette flow in the infinite channel}

\author[]{Qionglei Chen}
\address[]{Institute of Applied Physics and Computational Mathematics, 100088 Beijing, China}
\email{chen\_qionglei@iapcm.ac.cn}

\author[]{Zhen Li}

\address[]{School of Mathematical Sciences, Key Laboratory of Mathematics and Complex Systems, Ministry of Education, Beijing Normal University, 100875 Beijing, China}
\email{lizhen@bnu.edu.cn}

\author[]{Changxing Miao}
\address[]{Institute of Applied Physics and Computational Mathematics, 100088 Beijing, China}
\email{miao\_changxing@iapcm.ac.cn}

\keywords{Couette flow, infinite channel, optimal transition threshold}

\begin{abstract}
We investigate the stability of the 2-D Navier-Stokes equations in the infinite channel $\R\times [-1,1]$ with the Navier-slip boundary condition. We show that if the initial perturbations $\om^{in}$ around the Couette flow satisfy $\|\om^{in}\|_{H^3_{x,y}\cap L^1_x H^3_y}\leq c\nu^{\frac13}$, the solution admits enhanced dissipation at $x$-frequencies $|k|\gg \nu$ and inviscid damping effect. The key contributions lie in two parts: (1) we adopt the new decomposition of the vorticity $\om=\om_{L}+\om_e$, where $\om_L$ effectively captures a ``weak" enhanced dissipation $(1+\nu^{\frac13} t)^{-\frac14}e^{-\nu t}$ and the corresponding velocity exhibits the inviscid damping effect; (2) we introduce the dyadic decomposition for the long time scale $t\geq \nu^{-\frac16}$ and apply the ``infinite superposition principle" to the equation for $\om_e$ in order to control the growth induced by echo cascades, which appears to be novel and may hold independent significance.
\end{abstract}

\maketitle
\section{Introduction}\label{sec:Introduction}
We consider the incompressible Navier-Stokes equations in the infinite channel $\Omega=\R\times [-1,1]$:
\begin{equation}\label{NS}
	\left\{
    \begin{aligned}
    &\pat V + (V\cdot\na)V -\nu\De V+ \na P =0, \\
    &\na\cdot V = 0,\\
    & V(0,x,y) = V^{in}(x,y),
    \end{aligned}
	\right.
	\end{equation}
where $V(t,x,y)$ is the velocity field, $P$ is the corresponding pressure, and $\nu>0$ is the viscosity.

The hydrodynamic stability of the Navier-Stokes equation is a long-standing issue.
To resolve the Sommerfeld paradox \cite{Drazin-Reid}, Trefethen et al. \cite{Trefethen} initially reformulated the problem by studying the transition threshold. Bedrossian-Germain-Masmoudi \cite{Bedrossian-Germain-Masmoudi-4} presented a mathematical version:
Find a norm $\|\cdot\|_{X}$ and determine a $\gamma=\gamma(X)$ such that
\begin{align*}
&\|u^{in}\|_{X}\leq \nu^{\gamma}\Longrightarrow \text{ stability},\\
&\|u^{in}\|_{X}\gg \nu^{\gamma}\Longrightarrow \text{ instability}.
\end{align*}
The exponent $\gamma$ is referred to as the transition threshold.  Concerning the quantitative stability of shear flows, one can refer to \cite{Li-Wei-Zhang1,Wei-Zhang-Zhao-2,Chen-Wei-Zhang,Chen-Li-Wei-Zhang,Chen-Wei-Zhang2,Ding-Lin,Chen-Ding-Lin-Zhang, Zelati-Elgindi-Widmayer,Chen-Wei-Zhang1, Wei-Zhang-Zhao,Zelati-Gallay,Zelati,Deng-Wu-Zhang,Lin-Wu-Zhu,Ionescu-Jia,Ionescu-Jia1,Ionescu-Jia2,Ionescu-Jia3} and references therein.

The stability problem of the Couette flow in the infinite channel $\R\times[-1,1]\times\R$ has attracted the interest of many physicists. However, there has been little relatively mathematical research due to the complexity of analyzing the long-wave effect compared with the finite channel. For related studies in other domains, one can refer to \cite{Bedrossian-Germain-Masmoudi-1,Bedrossian-Germain-Masmoudi-2,Bedrossian-Germain-Masmoudi-3,Chen-Wei-Zhang2,Wei-Zhang-1} and references therein. To provide insights into the stability of the 3D Couette flow, numerous studies have focused on the 2D Couette flow, which serves as a simplified model of the 3D flow. These investigations help clarify the fundamental stability mechanisms and provide a cornerstone for understanding the 3D case. Notably, a series of stability results \cite{Bedrossian-Germain-Masmoudi-4,Bedrossian-He-Iyer-Wang,Bedrossian-He-Iyer-Wang1,Bedrossian-Masmoudi-Vicol,Bedrossian-Wang-Vicol,Masmoudi-Zhao} have been achieved for the 2D Couette flow in different domains and perturbation classes.

We present several current optimal transition threshold results for the Couette flow in different regions. In the domain $\T\times\R$, Masmoudi-Zhao \cite{Masmoudi-Zhao} and Wei-Zhang \cite{Wei-Zhang-2} obtained the optimal threshold $\gamma=\frac13$. For $\T\times [-1,1]$ with the non-slip boundary condition, Chen-Li-Wei-Zhang \cite{Chen-Li-Wei-Zhang} derived $\gamma\leq \frac12$. For the same domain $\T\times [-1,1]$ with the Navier-slip boundary condition, Wei-Zhang \cite{Wei-Zhang-3} recently proved the optimal threshold $\gamma= \frac13$ using the quasi-linear method. Arbon-Bedrossian \cite{Arbon-Bedrossian} initially studied the stability problem in unbounded domains such as $\R\times \R$,  $\R\times [0,\infty)$ and $\R\times [-1,1]$ with Navier-slip boundary conditions, showing the threshold $\gamma\leq \frac12+$. Very recently, Li-Liu-Zhao \cite{Li-Liu-Zhao} improved the threshold to $\frac13+$ in $\R\times \R$.

For the domain $\R\times [-1,1]$, it is natural to ask whether the optimal transition threshold $\frac13$ can be achieved. In this paper, we answer this question positively and obtain enhanced dissipation and inviscid damping.

Let $u=V-U$ be the perturbation of the velocity, which satisfies
\begin{equation}\label{pertu}
\left\{
\begin{aligned}
&\pat u+u\cdot\na u-\nu \De u+(u^{(2)},0)+y\pa_xu+\na p=0,\\
&\na\cdot u=0, \quad u^{(2)}(t,x,\pm1)=0,\quad \pa_y u^{(1)}(t,x,\pm1)=0,\\
&u(0,x,y)=u^{in}(x,y).
\end{aligned}
\right.
\end{equation}
The corresponding vorticity $\om=\pa_y u^{(1)}-\pa_x u^{(2)}$ can be written as
\begin{equation}\label{om}
\left\{
\begin{aligned}
&\pat\om+u\cdot\na \om-\nu\De \om+y\pa_x\om=0,\\
&\om(t,x,\pm1)=0,\quad \om\mid_{t=0}=\om^{in}(x,y).
\end{aligned}
\right.
\end{equation}

We state the main result of the paper.
\begin{theorem}\label{Th1}
Assume that $\om^{in}\in H^3_{x,y}\cap L^1_x H^3_y$ and $0<\nu< 1$. There exists a small constant $c>0$, independent of $\nu$, such that if
\begin{align*}
E_0:=\|\om^{in}\|_{H^3_{x,y}}+\|\om^{in}\|_{L^1_x H^3_y}
\leq c\nu^{\frac13},
\end{align*}
 the solution to the system \eqref{om} is global in time.

Moreover, let $\mathcal{M}=\mathcal{M}(D_x)$ and $\mathcal{M}_1(D_x)$ be the Fourier multipliers defined as
\begin{align}\label{oper M}
\mathcal{M}(k)=|k|^{\frac23}\chi+1-\chi, \quad \mathcal{M}_1 (k)=|k|^{\frac12}+|k|,
\end{align}
where $\chi$ is a cut-off function with $\chi(k)=1$ if $|k|\leq \frac12$ and $\chi(k)=0$ if $|k|\geq 1$.
There exist constants $ C>0$ and $\theta\in[0,\frac{1}{16})$ independent of $\nu$, such that the solution $\om$ satisfies the following global stability estimate:
\begin{align}\label{Th: esti om}
\|(1+\nu^{\frac13}t\mathcal{M} )^{\theta}\om\|_{ L^2_{x,y}}\leq &CE_0,
\end{align}
and inviscid damping estimate:
\begin{align}\label{inviscid}
\|(1+\nu^{\frac13} t\mathcal{M})^{\theta}\mathcal{M}_1 u\|_{L^2_tL^2_{x,y}}\leq C E_0.
\end{align}
\end{theorem}
\begin{remark}
If the initial data satisfy $\|\om^{in}\|_{H^3_{x,y}}+\|\lan \frac{1}{D_x}\ran^{s} \om^{in}\|_{L^2_{x,y}}\leq c\nu^{\frac13}$ with $s>\frac12$, then the equation \eqref{om} has the global solution that verifies \eqref{Th: esti om} and \eqref{inviscid}. This initial data condition implicitly imposes a smallness requirement on the low-frequency Fourier modes in the $x$-direction, i.e., for $|k|\rightarrow 0$,
$$|\widehat{\om^{in}}(k)|\approx O(|k|^{s-\varepsilon}),\quad \text{with } \varepsilon<1/2.$$
In contrast, for $\om^{in}\in H^3_{x,y}\cap L^1_x H^3_y$, the Fourier transform $\widehat{\om^{in}}$ is only bounded without additional low-frequency decay.
\end{remark}
\begin{remark}
We guess that the optimal transition threshold on the domain $\R\times \R^{+}$ can be obtained by applying a similar method.
\end{remark}
\begin{remark}
For shear flow $U(y)$ near Couette with $\|U(y)-y\|_{C^3}\leq c\nu^{\frac13}$, the same transition threshold and inviscid damping estimates as in Theorem \ref{Th1} remain valid.
\end{remark}
Let us briefly point out the new innovations:
\begin{itemize}
\item It is known that the presence of enhanced dissipation and inviscid damping essentially contributes to the transition threshold. However, the enhanced dissipation rate $e^{-\nu k^2 t^3}$ and the inviscid damping effect fail as the frequency $k\rightarrow 0$. This behavior poses an obstacle to obtain the optimal threshold on $\R\times [-1,1]$ by directly using the space-time estimates  and the energy method.
    To overcome this difficulty, motivated by \cite{Wei-Zhang-3}, we employ the quasi-linear method to decompose $\om=\om_{L}+\om_{e}$, where
\begin{itemize}
  \item The solution $\om_L$ to \eqref{equ: omL} exhibits ``weak" enhanced dissipation rate $(1+\nu t^3)^{-1/4}e^{-\nu t}$ and the velocity $u_{L}$ has the inviscid damping $(1+t)^{-1}$ in $\R\times[-1,1]$, which plays a crucial role in obtaining the optimal threshold.
   \item We re-decompose the error term $\om_e$ into $\om_e=\om_{1}+\om_{2}$ and introduce a novel space-time weighted space $X_{\th}$ to close the desired energy estimates. Here, we employ two frequency-dependent multipliers $\mathcal{M}$, $\mathcal{M}_1$ to better capture the solution's properties across distinct frequencies.
    To obtain the optimal space-time estimates of $\omega_2$,  we perform the dyadic decomposition $ \cup_{j\geq 1}(T_{j-1},T_j]$ with $T_0=\nu^{\frac16}$ and $T_j=2^{j}\nu^{-\frac13}$ for the long time scale $t> \nu^{-\frac16}$.
     Lastly, we apply the ``infinite superposition principle" to the equation of $\om_e$ in order to control the growth induced by echo cascades.
    \end{itemize}
\end{itemize}

Now, we specifically clarify the ideas. First, by invoking a ``generalized superposition principle",  we decompose $\om=\om_{L}+\om_{e}$, where
\begin{equation}\label{equ: omL}
\left\{
\begin{aligned}
&\pa_t\om_{L}+\nu(1-t^2\pa^2_x)\om_{L}+y\pa_x\om_{L}=0,\\
&\om_{L}\big|_{t=0}=\om^{in}(x,y),\quad \om_{L}(\pm1)=0,
\end{aligned}
\right.
\end{equation}
and
\begin{equation}\label{equ: ome}
\left\{
\begin{aligned}
&\pat\om_{e}-\nu \De\om_{e}+y\pa_x\om_e+u\cdot\na\om_{e}+u_{e}\cdot\na\om_{L}+E_r=0,\\
&u_{e}=(\pa_y,-\pa_x)\phi_e, \,\phi_e=\De^{-1}\om_e,\\
&\om_{e}(t,x,\pm1)=0,\, \om_{e}\big|_{t=0}=0,
\end{aligned}
\right.
\end{equation}
with
\begin{align*}
E_{r}=E_{r_{L}}+u_{L}\cdot{\na}\om_{L}, \quad E_{r_{L}}=\pa_t \om_{L}-\nu\De \om_{L}+y\pa_{x}\om_{L},\quad u_{L}=(\pa_y,-\pa_x)\De^{-1}\om_{L}.
\end{align*}

We analyze the construction of $\om_L$. The Couette flow is known to exhibit an exponential decay rate of $e^{-c \nu k^2 t^3}$, which can be captured naturally by the evolution equation
$$\pa_t f-\nu t^2 k^2 f=0.$$
However, this decay vanishes as $|k|\rightarrow 0$. Notably, for $|k|\leq \nu$ and $y\in[-1,1]$, the solution to
\begin{align*}
\pa_tf-\nu \pa^2_yf=0,
\end{align*}
with homogeneous boundary conditions decays as $e^{-\nu t}$. Motivated by these observations, we analyze the modified equation
\begin{align*}
\pa_t f+\nu(1+t^2k^2)f=0,
\end{align*}
whose solution simultaneously exhibits the decay properties at both high and low frequencies.
Moreover, the Couette flow satisfies the key identity
$$\pa_t[f(t,x-ty,y)]=[(\pa_t -y\pa_x)f](t,x-ty,y).$$
Based on these insights, we construct the linear equation of $\om_{L}$ (see \eqref{equ: omL}), which effectively describes the evolution of the initial data $\om^{in}$ while exhibiting enhanced dissipation and inviscid damping.

We now return to the equation of $\om_{L}$. Let $\widetilde{\om}_{L}(t,x,y)=\om_L(t,x+ty,y)$. we reformulate the problem as follows: \begin{equation*}
\left\{
\begin{aligned}
&(\pa_t+\nu(1-t^2\pa^2_x))\widetilde{\om}_{L}=0, \\
&\widetilde{\om}_{L}\big|_{t=0}=\om^{in}(x,y), \quad \widetilde{\om}_{L}(t,x\mp t,\pm1)=0.
\end{aligned}
\right.
\end{equation*}
whose solution possesses an explicit form
$$\mathcal{F}(\widetilde{\om}_{L})(k)=e^{-\nu t-\frac13\nu k^2 t^3}\hat{\om}^{in}(k,y).$$
The $L^p_k$ $(1\leq p<\infty)$ norm of $e^{-\frac13\nu k^2 t^3}$, which exhibits polynomial decay at low frequencies, prompts us to introduce $L^1_x$ as the initial data space. Moreover, we show that the solution $\om_L$ reveals additional polynomial decay $(1+\nu t^3)^{-\frac14}$ for $t\gg \nu^{-\frac13}$. The rate $e^{-\nu t}$ plays an important role at low frequencies to obtain the polynomial decay of $E_r$.

On the other hand, it remains challenging to obtain the estimates of $\om_e$ without explicit expression. For $t\leq \nu^{-\frac16}$, we use the energy method to derive the space-time estimate of $\om_e$. For $t\geq \nu^{-\frac16}$, the main difficulty lies in the polynomial growth of $\pa_y \om_{L}$. Motivated by \cite{Wei-Zhang-3}, we take advantage of the good property of $\om_L$ to decompose $\om_{e}=\om_{1}+\om_{2}$ such that
\begin{equation*}
\left\{
\begin{aligned}
&\pa_t\om_{1}-\nu\De\om_{1}+y\pa_x\om_{1}+u\cdot\na\om_e+u^{(1)}_e\pa_x\om_{L}+u^{(2)}_e(\pa_y+t\pa_x)\om_{L}+E_r=0, \\
&\om_{1}(t,x,\pm1)=0, \quad \om_{1}|_{t=T_0}=\om_{e}|_{t=T_0},
\end{aligned}
\right.
\end{equation*}
and
\begin{equation}\label{equ: om2 sec1}
\left\{
\begin{aligned}
&\pa_t\om_{2}-\nu\De\om_{2}+y\pa_x\om_{2}-tu^{(2)}_e\pa_x\om_{L}=0, \\
&\om_{2}(t,x,\pm1)=0, \quad \om_{2}|_{t=T_0}=0.
\end{aligned}
\right.
\end{equation}
Let us mention that the choice of working space and multipliers is adapted to fully match the decay properties of $\om_L$. More precisely, we introduce the spaces $X$ and $X_{\theta}$ equipped with
\begin{equation}\label{f tildeXnorm}
\begin{aligned}
\|f\|_{X}=&\|f\|_{L^\infty L^2}+\nu^{\frac12}\|\na f\|_{L^2 L^2}+\|\mathcal{M}_1\na \De^{-1}f\|_{L^2 L^2}+\nu^{\frac12}\|\pa_x f\|_{L^1 L^2},
\end{aligned}
\end{equation}
and
\begin{equation}\label{fXk}
\begin{aligned}
\|f\|_{X_{\theta}}=&\|(1+\nu^{\frac13}t\mathcal{M})^{\theta}f\|_{X},
\end{aligned}
\end{equation}
where $L^p L^q$ denotes $L^p((t_0,t_1);L^q_{x,y}(\R\times [-1,1]))$, and $\mathcal{M}$, $\mathcal{M}_1$ are defined as \eqref{oper M}. We point out that the distinct expression of $\mathcal{M}_1$ at high and low frequencies is crucial to obtain the inviscid damping estimates.

The space-time estimates for $\om_1$ can be derived from the decay estimate of $\om_L$.
As for $\om_2$, the equation \eqref{equ: om2 sec1} contains the reaction term that poses an obstacle to obtaining the optimal space-time estimate.
In fact, the decay rate $e^{-\nu k^2 t^3}$ of $\om_L$ approaches 1 and the $L^1_x$ initial data condition does not ensure the smallness of $\om_L$ when $|k|\rightarrow 0$. This leads to the failure of the reaction term to provide uniform control over integrability for $t\geq \nu^{-\frac{1}{6}}$.
To address this difficulty, we observe that the reaction term exhibits a temporal decay property at $T\approx 2^{j} \nu^{-\frac{1}{3}}$. This motivates the introduction of the dyadic decomposition for $t\geq \nu^{-\frac16}$. Specifically,
we define a cut-off function
\begin{equation}\label{def: chij}
\chi_j(t)=\left\{
\begin{aligned}
&1,\quad  t\in(T_{j-1}, T_j],\\
&0,\quad t\notin (T_{j-1}, T_j],
\end{aligned}
\right.
\qquad \text{with } T_0=\nu^{\frac16}, \, T_j=2^{j}\nu^{-\frac13}, \quad j\geq 1,
\end{equation}
and further decompose $\om_2=\sum_{j\geq 1}\om_{2,j}$, where
\begin{equation*}
\left\{
\begin{aligned}
&\pa_t \om_{2,j}-\nu\De\om_{2,j}+y\pa_x\om_{2,j}=tu^{(2)}_e\pa_x\om_{L}\chi_j, \quad t>T_{j-1}, \\
& \om_{2,j}(t,x,\pm1)=0, \quad \om_{2,j}|_{t=T_{j-1}}=0.
\end{aligned}
\right.
\end{equation*}
To gain a clearer understanding of $\om_{2,j}$, we shall present a brief analysis on its characteristics.
The evolution is divided into two time intervals: $(T_{j-1},T_{j}]$ and $(T_j,\infty]$. More precisely, $\om_{2,j}$ satisfies
\begin{equation*}
\left\{
\begin{aligned}
&\pa_t \om_{2,j}-\nu\De\om_{2,j}+y\pa_x\om_{2,j}=tu^{(2)}_e\pa_x\om_{L}, \quad T_{j-1}<t\leq T_j, \\
& \om_{2,j}(t,x,\pm1)=0, \quad \om_{2,j}|_{t=T_{j-1}}=0,
\end{aligned}
\right.
\end{equation*}
and
\begin{equation*}
\left\{
\begin{aligned}
&\pa_t \om_{2,j}-\nu\De\om_{2,j}+y\pa_x\om_{2,j}=0, \quad t>T_{j}, \\
& \om_{2,j}(t,x,\pm1)=0, \quad \om_{2,j}|_{t=T_{j}}=\om_{2,j}(T_j).
\end{aligned}
\right.
\end{equation*}
The following table displays the specific expression of $\om_{2}$.
\begin{center}
\begin{tabular}{|c|c|c|c|c|c|c|}
	\hline $T$ & $ (T_0, T_1]$ &  $ (T_1, T_2]$ &$(T_2,T_3]$ & \,$.\,.\,.$\, & $(T_{j-1},T_{j}]$  & \,.\,.\,.\,\\
   \hline  $\om_2$ & $ \om_{2,1}$&$ \om_{2,1}+\om_{2,2}$  &$\om_{2,1}+\om_{2,2}+\om_{2,3}$ &\,.\,.\,.&\, $ \sum_{1\leq i\leq j}\om_{2,i}$  &\,.\,.\,.\,\\
   \hline
\end{tabular}
\end{center}
Since the external force term vanishes in $(T_{j},\infty]$, the estimate of $\om_{2,j}$ on this interval is controlled by its initial value $\om_{2,j}(T_{j})$.
To provide a more intuitive understanding, we exhibit the following figure.

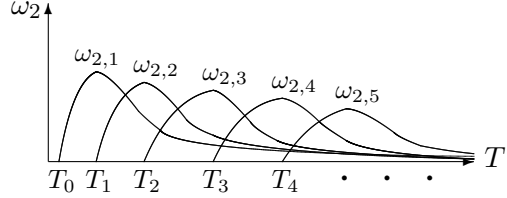
\setlength{\unitlength}{0.7cm}
\begin{picture}(6,4)(-6,-0.8)
\put(0,0){\vector(1,0){8}}
\put(8.2,-0.1){$T$}
\put(0,0){\vector(0,1){3}}
\put(-0.7,2.8){$\om_2$}
\qbezier(0.2,0)(0.45,1.5)(0.9,1.7)
\qbezier(0.9,1.7)(1.15,1.68)(1.7,0.95)
\qbezier(1.7,0.95)(1.9,0.75)(2.1,0.6)
\qbezier(2.1,0.6)(2.4,0.15)(8,0.1)
\put(0,-0.5){{\small$T_0$}}
\put(0.5,1.9){{\small$\omega_{2,1}$}}

\qbezier(0.9,0)(1.2,1.3)(1.8,1.5)
\qbezier(1.8,1.5)(2.1,1.45)(2.6,0.79)
\qbezier(2.6,0.79)(2.75,0.63)(3.1,0.5)
\qbezier(3.1,0.5)(3.8,0.15)(8,0.05)
\put(0.7,-0.5){{\small$T_1$}}
\put(1.6,1.7){{\small$\omega_{2,2}$}}

\qbezier(1.8,0)(2.3,1.2)(3.1,1.35)
\qbezier(3.1,1.35)(3.3,1.33)(3.8,0.75)
\qbezier(3.8,0.75)(3.9,0.6)(4.3,0.4)
\qbezier(4.3,0.4)(4.75,0.15)(8,0.05)
\put(1.6,-0.5){{\small$T_2$}}
\put(2.9,1.55){{\small$\omega_{2,3}$}}

\qbezier(3.1,0)(3.6,1.1)(4.4,1.2)
\qbezier(4.4,1.2)(4.7,1.15)(5.1,0.75)
\qbezier(5.1,0.75)(5.2,0.65)(5.6,0.4)
\qbezier(5.6,0.4)(6,0.15)(8,0.05)
\put(2.9,-0.5){{\small$T_3$}}
\put(4.2,1.4){{\small$\omega_{2,4}$}}

\qbezier(4.4,0)(4.9,0.8)(5.6,1)
\qbezier(5.6,1)(5.9,0.995)(6.5,0.55)
\qbezier(6.5,0.55)(6.6,0.45)(7,0.3)
\qbezier(7,0.3)(7.5,0.2)(8,0.15)
\put(4.2,-0.5){{\small$T_4$}}
\put(5.4,1.2){{\small$\omega_{2,5}$}}
\put(5.5,-0.35){\textbf{.}}
\put(6.3,-0.35){\textbf{.}}
\put(7.1,-0.35){\textbf{.}}
\end{picture}

\noindent Finally, the infinite superposition principle is applied to derive the space-time estimates for $\om_{2}$.

\vspace{0.2cm}

\noindent{\bf Notations.} Throughout this paper, we denote by $C$ a general constant independent of $\nu, k$ and  may vary from line to line.
We will use the following notations:
 $A\lesssim B $ stands for $ A\leq CB$,
 $\lan k\ran =\sqrt{1+|k|^2}$, $\hat{f}=\int_{\R}e^{-ikx} f(x)dx$ and
$$\|f\|_{L^q}=\|f\|_{L^q_{x,y}(\R\times (-1,1))},\quad \|f\|_{L^pL^q}=\|f\|_{L^p_tL^q_{x,y}(\R\times (-1,1))}, \quad \forall p,q\in[1,\infty].$$
\section{A priori estimates}
In this section, we give some space-time estimates for the solution to
\begin{equation}
\left\{
\begin{aligned}\label{equ: fg}
&\pa_t f-\nu \De f+y \pa_xf=g,\\
& f(t,x,\pm1)=0, \quad f|_{t=t_0}=f(t_0).
\end{aligned}
\right.
\end{equation}
In the following, we denote $f(t,k,y)=\int_{\R}e^{-ikx}f(t,x,y)dx$ and $g(t,k,y)=\int_{\R}e^{-ikx}g(t,x,y)dx$ for convenience.

For $\nu=0$ and $f(t_0)=0$, applying the Fourier transform to $x$-variable in \eqref{equ: fg} yields
\begin{equation}
\left\{
\begin{aligned}\label{equ:fky}
&\pa_tf(t,k,y)+ikyf(t,k,y)=g(t,k,y), \\
& f(t,k,\pm1)=0, \quad f(t_0)=0.
\end{aligned}
\right.
\end{equation}
\begin{lemma}\label{lemma: om decom}
Let $f(t,k,y)$ be the solution of  \eqref{equ:fky} for $t\in[t_0,t_1]$. Assume that $ k, l\in\R\setminus \{0\}$, $ k\neq l$, $e^{ilty}g\in L^2(t_0,t_1;H^1_0(-1,1))$ and $\psi=-(\pa^2_y-k^2)^{-1}f$. Then we have
\begin{equation}\label{esti: w+psi}
\|f\|^2_{L^\infty_t L^2_y}+(|k|+|k|^2)\|(\pa_y, k)\psi\|^2_{L^2_t L^2_y}\leq C
 | k-l|^{-1}\lan k-l \ran^{-1} \|(\pa_y, \lan k-l\ran )(e^{il yt}g)\|^2_{L^2_tL^2_y}.
\end{equation}
Moreover, for $s\in[0,\frac32)$, we have
\begin{equation}\label{esti: pay+kwL2L2}
\begin{aligned}
&\nu\|(1+ \nu^{\frac13} t)^{-s}(\pa_y, k)f\|^2_{L^2_tL^2_y}\\
\leq & C(1+\nu^{\frac13} (t_1-t_0))^{(3-2s)}\frac{|k-l|^2+|l|^2}{|k-l|\lan k-l\ran}\|(\pa_y, \lan k-l \ran)(e^{il yt}g)\|^2_{L^2_tL^2_y}.
\end{aligned}
\end{equation}
\end{lemma}
\begin{proof}
We first prove the case for $|k-l|\leq 1$, which will be divided into three steps.

\noindent\textbf{Step 1.} The estimate of $\|f\|_{L^\infty_t L^2_y}$.

Thanks to $f(t,k,\pm1)=g(t,k,\pm1)=0$, here, we introduce the Fourier transform in $y$ as
\begin{align*}
\widetilde{f}(t,\zeta)=\frac{1}{\sqrt{2\pi}}\int^{1}_{-1}f(t,y)e^{-iy\zeta}dy,\quad \tilde{g}(t,\zeta)=\frac{1}{\sqrt{2\pi}}\int^{1}_{-1}g(t,y)e^{-iy\zeta}dy.
\end{align*}
By Plancherel's formula, we have $\|\widetilde{f}(t)\|_{L^2_\zeta(\R)}=\|f(t)\|_{L^2_y(\textrm{I})}$ and
\begin{equation}\label{equality; leq1}
\begin{aligned}
\|(\pa_y, 1)(e^{il yt}g)\|_{L^2_tL^2_y}=&\int^{t_1}_{t_0}\int_{\R}(|\zeta+l t|^2+1)|\tilde{g}(t,\zeta)|^2d\zeta dt\\
=&\int^{t_1}_{t_0}\int_{\R}(|\zeta- k t+l t|^2+1)|\tilde{g}(t,\zeta- k t)|^2d\zeta dt.
\end{aligned}
\end{equation}
Notice that $\pa_t(e^{ikyt}f)=e^{ikyt}g$ and $f(t_0)=0$. Taking the Fourier transform in $y$-variable, we get
\begin{align*}
\frac{d}{dt}\widetilde{f}(t,\zeta- k t)=\tilde{g}(t,\zeta- k t), \quad \widetilde{f}(t_0,\zeta- k t_0)=0,
\end{align*}
which yields
\begin{align*}
\widetilde{f}(t,\zeta- k t)=\int^{t}_{t_0} \tilde{g}(\tau,\zeta- k \tau)d\tau.
\end{align*}
Let $F(\zeta)=\sup_{t\in[t_0,t_1]}|\widetilde{f}(t,\zeta- k t)|$ for fixed $\zeta\in\R$.
We use H\"older's inequality to obtain
\begin{align*}
F(\zeta)\leq &\|\tilde{g}(t,\zeta- k t)\|_{L^1_t(t_0,t_1)}\\
\leq&\|(1+|\zeta- k t+l t|^2)\tilde{g}(t,\zeta- k t)\|_{L^2_t(t_0,t_1)}\|(1+|\zeta- k t+l t|^2)^{-1}\|_{L^2_t(t_0,t_1)}.
\end{align*}
Due to
\begin{align*}
\int_{\R}\frac{1}{1+|\zeta- k t+l t|^2}dt\leq \frac{1}{| k-l|}\int_{\R}\frac{1}{1+|t|^2}dt=\frac{\pi}{| k-l|},
\end{align*}
we have
\begin{align*}
|F(\zeta)|^2\leq \frac{\pi}{| k-l|}\|(1+|\zeta- k t+l t|^2)\tilde{g}(t,\zeta- k t)\|^2_{L^2_t(t_0,t_1)},
\end{align*}
which together with \eqref{equality; leq1} implies
\begin{align}\label{esti: W}
\int_{\R}|F(\zeta)|^2 d\zeta\leq \frac{\pi}{| k-l|}\|(\pa_y, 1)(e^{il yt}g)\|^2_{L^2_tL^2_y}.
\end{align}

For every $t\in[t_0,t_1]$, there holds
\begin{align}\label{esti: omL2I}
\|f(t)\|^2_{L^2_y(I)}=\|\widetilde{f}(t)\|^2_{L^2_\zeta(\R)}=\int_{\R}|\widetilde{f}(t,\zeta- k t)|^2 d\zeta\leq \int_{\R}|F(\zeta)|^2d\zeta,
\end{align}
which gives
$$\|f(t)\|^2_{L^\infty_t L^2_y}\leq \frac{\pi}{| k-l|}\|(\pa_y, 1)(e^{il yt}g)\|^2_{L^2_tL^2_y}.$$
\textbf{Step 2.} The estimate of $\|(\pa_y,k)\psi\|_{L^2_t L^2_y}$.

For $|k|\leq 1$, let
\begin{align*}
\psi_{*}(t,k,y)=\frac{1}{\sqrt{2\pi}}\int_{\R}\frac{\widetilde{f}(t,k,\zeta)}{\zeta^2+ 1}e^{iy\zeta} d\zeta.
\end{align*}
We have $\psi_{*}(y)\in H^2(\R)$ and $-(\pa^2_y-1)\psi_{*}=f$ for $y\in[-1,1]$. Thanks to $\psi=-(\pa^2_y-k^2)^{-1}f$, we get
\begin{align*}
\|(\pa_y, k)\psi(t)\|^2_{L^2_y}=&\lan\psi(t),f(t)\ran=\lan \psi(t),-(\pa^2_y-1)\psi_{*}(t)\ran\\
=&\lan \pa_y\psi(t),\pa_y\psi_{*}(t)\ran +\lan \psi(t),\psi_{*}(t)\ran\leq \|(\pa_y, 1)\psi(t)\|_{L^2_y}\|(\pa_y, 1)\psi_{*}\|_{L^2_y}.
\end{align*}
Due to $\|\psi\|_{L^2_y}\leq \|\pa_y \psi\|_{L^2_y}$, we have
\begin{align*}
\|(\pa_y, k)\psi\|^2_{L^2_t L^2_y}\leq \|(\pa_y, 1)\psi_{*}\|^2_{L^2_t L^2_y}=\int^{t_1}_{t_0}\int_{\R}\frac{|\widetilde{f}(t,\zeta)|^2}{|\zeta|^2+1}d\zeta dt= \int^{t_1}_{t_0}\int_{\R}\frac{|\widetilde{f}(t,\zeta- k t)|^2}{|\zeta- k t|^2+1} d\zeta dt.
\end{align*}
Using \eqref{esti: W}, we then obtain
\begin{equation}\label{psiL2L2kleq1}
\begin{aligned}
| k|\|(\pa_y, k)\psi\|^2_{L^2_tL^2_y}=& \int^{t_1}_{t_0}\int_{\R}\frac{| k||\widetilde{f}(t,\zeta- k t)|^2}{|\zeta- k t|^2+1} d\zeta dt\leq \int^{t_1}_{t_0}\int_{\R}\frac{| k||F(\zeta)|^2}{|\zeta- k t|^2+1} d\zeta dt\\
\leq&\int_{\R}|F(\zeta)|^2\int^{t_1}_{t_0}\frac{| k|}{|\zeta- k t|^2+1}dtd\zeta
\leq \frac{\pi}{| k-l|}\|(\pa_y, 1)(e^{il yt}g)\|^2_{L^2_tL^2_y}.
\end{aligned}
\end{equation}

For $|k|\geq 1$, let
$\psi_{*}(t,k,y)=\frac{1}{\sqrt{2\pi}}\int_{\R}\frac{\widetilde{f}(t,k,\zeta)}{\zeta^2+ k^2}e^{iy\zeta} d\zeta.$
In some way, as in the above deriving
\begin{align*}
| k|^2\|(\pa_y, k)\psi\|^2_{L^2_tL^2_y}\leq \frac{\pi}{| k-l|}\|(\pa_y, 1)(e^{il yt}g)\|^2_{L^2_tL^2_y}.
\end{align*}
\textbf{Step 3.} The estimate of $\|(\pa_y,k)f\|_{L^2_tL^2_y}$.

Thanks to $\pa_t f+ikyf=g$ and $f(t_0)=0$, we have
\begin{align*}
\pa_t(e^{il yt}f)+i( k-l)y(e^{il yt}f)=e^{il yt}g,\quad e^{ily t_0}f(t_0)=0.
\end{align*}
Taking the derivative of $y$, we get
\begin{align*}
(\pa_t+i( k-l)y)\pa_y(e^{il yt}f)+i( k-l)(e^{il yt}f)=\pa_y (e^{il yt}g),\quad \pa_y(e^{il y t_0}f(t_0))=0,
\end{align*}
which implies
\begin{align*}
e^{i(k-l)yt}\pa_y(e^{il yt}f)+\int^{t}_{t_0}i(k-l)e^{iky\tau}f-e^{i(k-l)y\tau}\pa_y(e^{il y\tau}g(\tau)d\tau=0.
\end{align*}
Taking the $L^2_y$ norm, we obtain
\begin{align*}
\|\pa_y(e^{il yt}f(t))\|_{L^2_y}\leq &\int^{t}_{t_0}\|i( k-l)e^{il y\tau}f(\tau)-\pa_y(e^{il y\tau}g(\tau))\|_{L^2_y}d\tau\\
\leq&| k-l|t\|f\|_{L^\infty_t L^2_y}+t^{\frac12}\|\pa_y(e^{il yt}g)\|_{L^2_tL^2_y}.
\end{align*}
Due to $\pa_yf=e^{- il yt}\pa_y(e^{il yt}f)-il tf$, we have
\begin{align*}
\|\pa_y f\|_{L^2_y}\leq& \|\pa_y(e^{il yt}f(t))\|_{L^2_y}+|l t|\|f(t)\|_{L^2_y}\\
\leq&(| k-l |+|l|)t\|f\|_{L^\infty_t L^2_y}+t^{\frac12}\|\pa_y(e^{il yt}g)\|_{L^2_tL^2_y}.
\end{align*}
Furthermore, we deduce
\begin{align*}
&\nu\|(1+ \nu^{\frac13}t)^{-s}\pa_yf\|^2_{L^2_tL^2_y}+\nu| k|^2\|(1+ \nu^{\frac13}t)^{-s}f\|^2_{L^2_tL^2_y}\\
\leq&\nu(| k-l|+|l|)^2\|(1+ \nu^{\frac13}t)^{-s}t\|^2_{L^2(t_0,t_1)}\|f\|^2_{L^\infty_t L^2_y}\\
&+\nu\|(1+ \nu^{\frac13}t)^{-s}t^{\frac12}\|^2_{L^2(t_0,t_1)} \|\pa_y(e^{il yt}g)\|^2_{L^2_tL^2_y}
+\nu| k|^2\|(1+\nu^{\frac13}t )^{-s}\|^2_{L^2(t_0,t_1)}\|f\|^2_{L^\infty_t L^2_y}\\
\lesssim & (1+\nu^{\frac13} t_1)^{(3-2s)}((| k-l|+|l|)^2 +\nu^{2/3}| k|^2)\|f\|^2_{L^\infty_t L^2_y}\\
&+(1+\nu^{\frac13} t_1)^{(2-2s)}\nu^{1/3}\|\pa_y(e^{il yt}g)\|^2_{L^2_tL^2_y}.
\end{align*}
Inserting \eqref{esti: w+psi} into the above inequality, we obtain
\begin{align*}
&\nu\|(1+ \nu^{\frac13}t)^{-s}(\pa_y,k)f\|^2_{L^2_tL^2_y}\\
\leq &C(1+\nu^{\frac13} t_1)^{(3-2s)}(| k-l|+|l|^2| k-l|^{-1})\|(\pa_y, 1)(e^{il yt}g)\|^2_{L^2_tL^2_y}.
\end{align*}

For $|k-l|\geq 1$, noticing that
\begin{equation*}
\begin{aligned}
&\|(\pa_y, k-l)(e^{il yt}g)\|_{L^2_tL^2_y}
=\int^{t_1}_{t_0}\int_{\R}(|\zeta- k t+l t|^2+|k-l|^2)|\tilde{g}(t,\zeta- k t)|^2d\zeta dt,\\
&\int_{\R}\frac{1}{|\zeta- k t+l t|^2+|k-l|^2}dt\leq \frac{1}{|k-l|^2}\int_{\R}\frac{1}{1+|t|^2}dt=\frac{\pi}{|k-l|^2},
\end{aligned}
\end{equation*}
we can obtain \eqref{esti: w+psi} and \eqref{esti: pay+kwL2L2} by a similar calculation as step 1--step 3.
\end{proof}
Consider the homogeneous equation
\begin{equation}\label{equ: fk}
\left\{
\begin{aligned}
&\pa_t f(t,k,y)-\nu (\pa^2_y-k^2) f(t,k,y)+iky f(t,k,y)=0,\quad k\neq 0, \\
&f(t,k,\pm 1)=0, \quad f|_{t=0}=f(0).
\end{aligned}
\right.
\end{equation}
The solution $f(t,k,y)$ satisfies the following estimates.
\begin{lemma}\label{prop: f esti}
Let $f$ solve \eqref{equ: fk}. There hold the enhanced dissipation estimates
\begin{align}
\|k f\|_{L^2_y}\leq& C(\nu t)^{-\frac12}(1+t)^{-1}\|f(0)\|_{L^2_y}, \label{enhan1}\\
\| f\|_{L^2_y}\leq& C(1+\nu t+\nu k^2t^3)^{-\frac12}\|f(0)\|_{L^2_y}.\label{enhan2}
\end{align}
Integrating with $k$, there hold
\begin{align}
\|k f\|_{L^2_kL^2_y}\leq& C(\nu t)^{-\frac12}(1+t)^{-1}\|f(0)\|_{L^2_kL^2_y},\label{enhan3}\\
\|\mathcal{M}(k) f\|_{L^2_kL^2_y}\leq& C\nu^{-\frac13}(1+t)^{-1}(1+\nu t)^{-\frac16}\|f(0)\|_{L^2_kL^2_y},\label{enhan4}
\end{align}
where $\mathcal{M}$ is defined in \eqref{oper M}.
\end{lemma}
\begin{proof}
We adopt the hypocoercivity argument. More precisely, we construct the energy functional
$$\Phi_k(t)=(1+\gamma_0\nu | k|^2 t^3)\|f\|^2_{L^2_y}+\alpha_0\nu t\|(\pa_y, k)f\|^2_{L^2_y}+\beta_0\nu t^2\re\lan ikf,\pa_y f\ran,$$
with $\gamma_0=1/600$, $\alpha_0=12\gamma_0$, $\beta_0=5\gamma_0$,
which satisfies
$$\Phi_k(t)\geq (1+\gamma_0\nu | k|^2 t^3/4)\|f\|^2_{L^2_y}+\alpha_0\nu t\|(\pa_y, k) f\|^2_{L^2_y}/4,\quad \frac{d}{dt}\Phi_k(t)\leq 0.$$
It follows that
\begin{align*}
(1+\gamma_0\nu | k|^2 t^3/4)\|f\|^2_{L^2_y}+\alpha_0\nu t\|(\pa_y, k) f\|^2_{L^2_y}/4\leq \Phi_k(0)=\|f(0)\|_{L^2},
\end{align*}
which yields \eqref{enhan1} and \eqref{enhan2}. For more details, please see \cite[Proposition 4.1]{Wei-Zhang-3}.

The estimate \eqref{enhan3} results from integrating both sides of \eqref{enhan1} with respect to $k$.
For \eqref{enhan4}, a direct computation shows
\begin{align*}
\|\mathcal{M}(k) f\|^2_{L^2_kL^2_y}\lesssim &\int_{\R}\frac{|k|^{\frac43}}{1+\nu t+\nu k^2 t^3}\|f(0)\|^2_{L^2_y}dk\\
\lesssim &\nu^{-\frac23}(1+t)^{-2}\int_{\R }\frac{1}{(1+\nu t+\nu k^2 t^3)^{\frac13}}\|f(0)\|^2_{L^2_y}dk\\
\lesssim & \nu^{-\frac23}(1+t)^{-2}(1+\nu t)^{-\frac13}\|f(0)\|^2_{L^2_k L^2_y}.
\end{align*}
Hence we complete the proof of Lemma \ref{prop: f esti}.
\end{proof}
Now, let us return to the inhomogeneous equation \eqref{equ: fg}.
\begin{lemma}\label{lemma: inviscid damping}
Let $f$ be the solution of \eqref{equ: fg}. Then we have
\begin{align}\label{fY esti}
\|f\|^2_{Y}\leq C\|f(0)\|^2_{L^2}+C\nu^{-1}\int^t_0\|\na\De^{-1} g\|^2_{L^2} ds,
\end{align}
where
\begin{align}\label{fY norm}
\|f\|_{Y}:=\|f\|_{L^\infty L^2}+\nu^{\frac12}\|\na f\|_{L^2L^2}+\|\mathcal{M}_1 \na \De^{-1}f\|_{L^2L^2}.
\end{align}
\end{lemma}
\begin{proof}
We first perform the Fourier transform on \eqref{equ: fg} with respect to $x$-variable to get
\begin{equation}
\left\{
\begin{aligned}\label{equ: f Four}
&\pa_t f(t,k,y)-\nu(\pa^2_y-k^2)f(t,k,y)+ikyf(t,k,y)=g(t,k,y),\\
&f|_{t=0}=f(0,k,y), \quad f(t,k,\pm1)=0.
\end{aligned}
\right.
\end{equation}
Taking the inner product with $f(t,k,y)$ for \eqref{equ: f Four}, we have
\begin{align*}
\frac{d}{dt}\|f\|^2_{L^2_y}=&2\re\lan \nu (\pa^2_y-k^2)f-iky f+g, f\ran =2 \re\lan (\pa^2_y-k^2) (\nu f+ (\pa^2_y-k^2)^{-1}g),f\ran\\
=& -2\nu\|(\pa_y, k)f\|^2_{L^2_y}-2\re\lan \pa_y (\pa^2_y-k^2)^{-1}g,\pa_y f\ran-2| k|^2\re\lan (\pa^2_y-k^2)^{-1}g,f\ran\\
\leq&-2\nu\|(\pa_y, k)f\|^2_{L^2_y}+2 \|(\pa_y, k)f\|_{L^2_y}\|(\pa_y, k)(\pa^2_y-k^2)^{-1}g\|_{L^2_y}\\
\leq&-\nu\|(\pa_y, k)f\|^2_{L^2_y}+\nu^{-1}\|(\pa_y, k)(\pa^2_y-k^2)^{-1}g\|^2_{L^2_y}.
\end{align*}
Integrating with respect to $t$ and $k$ and then using Plancherel's formula, we derive
\begin{align}\label{f energy norm}
\|f(t)\|^2_{L^2}+\nu\int^t_0\|\na f(s)\|^2_{L^2} ds\leq \|f(0)\|^2_{L^2}+\nu^{-1} \int^t_0\|\na  \De^{-1}g\|^2_{L^2} ds.
\end{align}

It remains to bound $\|\mathcal{M}_1 \na \De^{-1}f\|_{L^2L^2}$.
Define $\mathfrak{J}_k$ the singular integral operator that
\begin{align*}
\mathfrak{J}_k[f](y)=| k|p.v. \frac{ k}{| k|}\int^1_{-1}\frac{1}{2i(y-y')}G_k(y,y')f(y')dy', \quad k\neq 0,
\end{align*}
with
\begin{equation*}
G_{k}(y,y')=-\frac{1}{ k\sinh 2 k}
\left\{
\begin{aligned}
&\sinh ( k(1-y'))\sinh( k(1+y)), \quad y\leq y',\\
&\sinh( k(1-y))\sinh( k(1+y')), \quad y\geq y'.
\end{aligned}
\right.
\end{equation*}
Since $\mathfrak{J}_k$ is symmetric, we have
\begin{align*}
\frac12\frac{d}{dt}\re\lan f,\mathfrak{J}_k[f]\ran_y=\re\lan \frac{d}{dt} f,\mathfrak{J}_k[f]\ran_y&=\re\lan-ikyf+\nu (\pa^2_y-k^2)f+g, \mathfrak{J}_k[f] \ran_y\\
&=:\textrm{I}_1+\textrm{I}_2+\textrm{I}_3.
\end{align*}
Let $\psi=(\pa^2_y-k^2)^{-1}f(t,k,y)$. For $\textrm{I}_1$, we use the symmetry of $\mathfrak{J}_k$ again to obtain
\begin{align*}
2\textrm{I}_1=&\re\lan\mathfrak{J}_k[-ikyf],f\ran_y+ \re\lan \mathfrak{J}_k[f],-ikyf\ran_y \\
=&-\re\int^1_{-1}| k|p.v. \Big(\int_{\R}\operatorname{sgn}( k)\frac{G_k(y,y')}{2(y-y')} k y'f(y')\overline{f}(y)dy'\Big)dy\\
&+\re\int^1_{-1}\Big(| k|p.v. \int^1_{-1}\operatorname{sgn}( k)\frac{G_k(y,y')}{2(y-y')}f(y')ky\overline{f}(y)dy'\Big)dy\\
=&\frac{| k|^2}{2}\int^1_{-1}p.v.\Big(\int^1_{-1}\frac{G_k(y,y')}{(y-y')}(y-y')f(y')dy'\Big)\overline{f}(y)dy\\
=&\frac{| k|^2}{2}\lan(\pa^2_y-k^2)^{-1}f,f\ran=\frac{| k|^2}{2}\lan\psi,(\pa^2_y-k^2) \psi\ran=-\frac{| k|^2}{2}\|(\pa_y, k)\psi\|^2_{L^2_y},
\end{align*}
which is
$$\textrm{I}_1=-\frac{| k|^2}{4}\|(\pa_y, k)\psi\|^2_{L^2_y}.$$

For $\textrm{I}_2$, integrating by parts and using the condition $\mathfrak{J}_k [f]|_{y=\pm1}=0$, we have
\begin{equation}\label{esti: T2}
\begin{aligned}
\textrm{I}_2=\nu \re\lan (\pa^2_y-k^2) f,\mathfrak{J}_k[f]\ran_y=&-\nu\re\lan \pa_y f,\pa_y\mathfrak{J}_k [f]\ran_y-\nu| k|^2\re\lan f,\mathfrak{J}_k [f]\ran_y\\
\leq&\nu\|(\pa_y, k)f\|_{L^2_y}\|(\pa_y, k)\mathfrak{J}_k[f]\|_{L^2_y}.
\end{aligned}
\end{equation}
In light of Lemma \ref{operator estimate} and Lemma 7.2 in \cite{Arbon-Bedrossian}, we deduce
\begin{equation}\label{payJkfL2}
\begin{aligned}
\|(\pa_y, k)\mathfrak{J}_k[f]\|_{L^2_y}\leq& \|[\pa_y, \mathfrak{J}_k][f]\|_{L^2_y}+\|\mathfrak{J}_k[\pa_y f]\|_{L^2_y}+\|k\mathfrak{J}_k[f]\|_{L^2_y}\\
\leq &C(| k|\|f\|_{L^2_y}+\min\{1,|k|\}\|\pa_y f\|_{L^2_y})\\
\leq &C\min\{1,|k|\}\|(\pa_y,k) f\|_{L^2_y}.
\end{aligned}
\end{equation}
Inserting \eqref{payJkfL2} into \eqref{esti: T2}, we derive that
\begin{align*}
\textrm{I}_2\leq C\min\{1,|k|\}\nu\|(\pa_y, k)f\|^2_{L^2_y}.
\end{align*}
In a similar way, we have
\begin{align*}
\textrm{I}_3\leq C\min\{1,|k|\} (\nu^{-1}\|(\pa_y, k)(\pa^2_y-k^2)^{-1}g\|^2_{L^2_y}+ \nu\|(\pa_y, k)f\|^2_{L^2_y}).
\end{align*}
Combining the estimates $\textrm{I}_1-\textrm{I}_3$, we obtain
\begin{align*}
&\frac12\frac{d}{dt}\re\lan f, \mathfrak{J}_k [f]\ran\\
\leq &-\frac{| k|^2}{4}\|(\pa_y, k)\psi\|^2_{L^2_y}+C \min\{1,|k|\} (\nu\|(\pa_y, k)f\|^2_{L^2_y}+\nu^{-1}\|(\pa_y,k)(\pa^2_y-k^2)^{-1}g\|_{L^2_y}),
\end{align*}
which yields
\begin{equation}\label{pay kpsiL2L2}
\begin{aligned}
&\int^t_0| k|^2\|(\pa_y, k)\psi(s)\|^2_{L^2_y} ds\\
\lesssim &|\lan f(t), \mathfrak{J}_k[f(t)]\ran|+|\lan f(0), \mathfrak{J}_k[f(0)]\ran|+\nu \min\{1,|k|\} \int^t_0 \|(\pa_y, k)f\|^2_{L^2_y} ds\\
&+\nu^{-1} \min\{1,|k|\} \int^t_0 \|(\pa_y, k)(\pa_y^2-k^2)^{-1}g\|^2_{L^2_y}ds.
\end{aligned}
\end{equation}
By Lemma \ref{operator estimate}, we have
\begin{align*}
|\lan f(t), \mathfrak{J}_k[f(t)]\ran|\leq &\|f(t)\|_{L^2_y}\|\mathfrak{J}_k[f(t)]\|_{L^2_y}\leq \min\{1,|k|\}\|f(t)\|^2_{L^2_y}, \\
 |\lan f(0), \mathfrak{J}_k[f(0)]\ran|\leq &\|f(0)\|_{L^2_y}\|\mathfrak{J}_k[f(0)]\|_{L^2_y}\leq\min\{1,|k|\}\|f(0)\|^2_{L^2_y}.
\end{align*}
Inserting the above two estimates into \eqref{pay kpsiL2L2}, we arrive at
\begin{align*}
&\int^t_0(|k|+| k|^2)\|(\pa_y, k)\psi(s)\|^2_{L^2_y} ds\\
\lesssim &\|f(0)\|^2_{L^2_y}+\|f(t)\|^2_{L^2_y}
+\nu^{-1} \int^t_0\|(\pa_y, k)(\pa_y^2-k^2)^{-1}g\|^2_{L^2_y}ds+ \nu\int^t_0 \|(\pa_y, k)f\|^2_{L^2_y}ds.
\end{align*}
Integrating with $k$ and using Plancherel's formula, we then use \eqref{f energy norm} to obtain
\begin{align*}
\|\mathcal{M}_1 \na \psi(s)\|^2_{L^2L^2}
\lesssim &\|f(0)\|^2_{L^2}+\|f(t)\|^2_{L^2}+ \nu \|\na f\|^2_{L^2L^2}+\nu^{-1} \|\na \De^{-1}g\|^2_{L^2L^2}\\
\lesssim & \|f(0)\|^2_{L^2}+\nu^{-1} \|\na \De^{-1}g\|^2_{L^2L^2},
\end{align*}
 which together with \eqref{f energy norm} completes the proof of Lemma \ref{lemma: inviscid damping}.
\end{proof}
\begin{proposition}\label{prop: fX estimate}
Let $f$ be the solution of \eqref{equ: fg} for $t\in [t_0,t_1]$. For $X$ defined in \eqref{f tildeXnorm}, there exists a suitable small constant $\epsilon$ such that
\begin{align}
\|f(t)\|_{X}\leq &C\|f(t_0)\|_{L^2}+C\| g\|_{L^1L^2},\label{esti: ftildeX} \\
\nu^{\frac13}\|\mathcal{M} f\|_{L^1 L^2}\leq &C_{\epsilon}\nu^{-\epsilon}\|f(t_0)\|_{L^2}+C\|g\|_{L^1 L^2},\label{f enhance}
\end{align}
Let $X_{\theta}$-norm be defined in \eqref{fXk} with $\theta\geq 0$. For $t_0\geq \nu^{-1-\epsilon}$, we have
\begin{align}
\|f\|_{X_{\theta}}\leq &C\|(1+\nu^{\frac13} t_0 \mathcal{M})^{\theta}f(t_0)\|_{L^2}+C\|(1+ \nu^{\frac13}t\mathcal{M})^{\theta}g\|_{L^1 L^2},\label{fX large time}\\
\nu^{\frac13}\|\mathcal{M}(1+ \nu^{\frac13}t\mathcal{M})^{\theta}f\|_{L^1L^2}\leq &C\|(1+\nu^{\frac13} t_0 \mathcal{M})^{\theta}f(t_0)\|_{L^2}+C\|(1+ \nu^{\frac13}t\mathcal{M})^{\theta}g\|_{L^1 L^2}.\label{esti: nu1/3M f}
\end{align}
\end{proposition}
\begin{proof}
\textbf{Step 1.} The estimate of homogeneous equation.

Let $S(t,s)$ be the solution operator of
\begin{equation*}
\left\{
\begin{aligned}
&\pa_t f-\nu\De f+y\pa_x f=0,\\
&f(t,x,\pm1)=0,\quad  f|_{t=s}=h(x,y).
\end{aligned}
\right.
\end{equation*}
For $t\geq s$, we have $f(t,x,y)=S(t,s) h(x,y)$. It follows from Lemma \ref{lemma: inviscid damping} with $g=0$ and $s=0$ that
\begin{align}\label{energy esti}
\|S(t,0)h\|_{Y} \leq C\|h\|_{L^2}.
\end{align}
Thanks to \eqref{enhan3} and \eqref{enhan2}, we have
\begin{align}
\int^{\infty}_0\|\pa_x S(t,0)h\|_{L^2} dt\lesssim& \int^{\infty}_0(\nu t)^{-\frac12}(1+t)^{-1}dt\|h\|_{L^2}\lesssim \nu^{-\frac12}\|h\|_{L^2},\label{enhan esti1}\\
\int^{\infty}_{0}\| S(t,0)h(k,y)\|_{L^2_y} dt\lesssim &\int^{\infty}_{0}(1+\nu k^2 t^3)^{-\frac12}dt\|h(k,y)\|_{L^2_y}
\lesssim \nu^{-\frac13}|k|^{-\frac23}\|h(k,y)\|_{L^2_y}.\label{enhan esti2}
\end{align}
By \eqref{enhan4} and
$$\nu^{-\frac13}(1+t)^{-1}(1+\nu t)^{-\frac16}\leq \nu^{-\frac13-\epsilon}(1+t)^{-1-\epsilon},$$
we deduce
\begin{align}
&\int^{\infty}_{0}\|\mathcal{M} S(t,0)h\|_{L^2} dt\lesssim \int^{\infty}_{0}\nu^{-\frac13-\epsilon}(1+t)^{-1-\epsilon}\|h\|_{L^2}dt\lesssim \epsilon^{-1}\nu^{-\frac13-\epsilon}\|h\|_{L^2},\label{pash1}\\
&\int^{\infty}_{\nu^{-1-\epsilon}}\|\mathcal{M}  S(t,0)h\|_{L^2} dt\lesssim \int^{\infty}_{\nu^{-1-\epsilon}}\nu^{-\frac13}(1+t)^{-1}(1+\nu t)^{-\frac16}\|h\|_{L^2}dt\lesssim   \nu^{-\frac13}\|h\|_{L^2}.\label{St0hL1L2}
\end{align}

For $t\geq s\geq 0$, it follows from \eqref{energy esti}--\eqref{pash1} that
\begin{equation}
\begin{aligned}\label{paxS(t,s)gL1L22}
\|S(t,s)h\|_{X}+ \nu^{\frac13}\Big\|\int^{\infty}_s\big\|\mathcal{M}  S(t,s) h\big\|_{L^2_y} dt\Big\|_{L^2_x}+\nu^{\frac13+\epsilon}\|\mathcal{M} S(t,s) h \|_{L^1L^2}\lesssim \|h\|_{L^2}.
\end{aligned}
\end{equation}
We use \eqref{energy esti}, \eqref{enhan esti1} and \eqref{St0hL1L2} to get
\begin{equation}
\begin{aligned}\label{paxStlarge}
\|S(t,s)h\|_{X}+\nu^{\frac13}\int^{\infty}_s\|\mathcal{M}  S(t,s) h\|_{L^2} dt\lesssim \|h\|_{L^2}, \quad s\geq \nu^{-1-\epsilon}.
\end{aligned}
\end{equation}

\textbf{Step 2.} The estimate of $f$.

Using the solution operator $S(t,s)$, the solution of \eqref{equ: fg} can be rewritten as
\begin{align}\label{solution of f}
f(t)=S(t,t_0)f(t_0)+\int^t_{t_0}S(t,s)g(s)ds.
\end{align}
Thanks to \eqref{paxS(t,s)gL1L22}, we directly obtain \eqref{esti: ftildeX} and \eqref{f enhance}.

For $t_0\geq \nu^{-1-\epsilon}$, we use \eqref{paxStlarge} to deduce
\begin{equation}\label{esti: fX+L1L2}
\begin{aligned}
\|f(t)\|_{X}+\nu^{\frac13}\|\mathcal{M} f\|_{L^1 L^2}
\lesssim  \|f(t_0)\|_{L^2}+\| g\|_{L^1L^2}.
\end{aligned}
\end{equation}
Let $\tilde{f}=(1+ \sigma\nu^{\frac13}t\mathcal{M})^{\theta}f$, where $\sigma\in(0,1)$ will be determined later. It satisfies
\begin{align*}
\pa_t \tilde{f}-\nu\De \tilde{f}+y\pa_x\tilde{f}=(1+\sigma \nu^{\frac13}t\mathcal{M})^{\theta} g+ \theta \sigma\nu^{\frac13}\mathcal{M} (1+\sigma \nu^{\frac13}t\mathcal{M})^{\theta-1}f.
\end{align*}
Noticing that
\begin{align*}
&\|\nu^{\frac13} \mathcal{M}(1+\sigma \nu^{\frac13}t\mathcal{M})^{\theta-1} f\|_{ L^2}=\|\nu^{\frac13} \mathcal{M}(1+\sigma \nu^{\frac13}t\mathcal{M})^{-1} \tilde{f}\|_{L^2}\\
=&\|\nu^{\frac13} \mathcal{M}(k) (1+\sigma \nu^{\frac13}t\mathcal{M}(k))^{-1}\|\tilde{f}\|_{L^2_y}\|_{L^2_k}
\leq \|\nu^{\frac13} \mathcal{M}(k) \|\tilde{f}\|_{L^2_y}\|_{L^2_k},
\end{align*}
we then use \eqref{esti: fX+L1L2} to get
\begin{align*}
&\|\tilde{f}\|_{X}+\nu^{\frac13}\|\mathcal{M} \tilde{f}\|_{L^1 L^2}\\
\leq &C(\|\tilde{f}(t_0)\|_{L^2}+\|(1+\sigma \nu^{\frac13}t\mathcal{M})^{\theta} g\|_{L^1 L^2}+\sigma\theta\|\nu^{\frac13} \mathcal{M}(1+\sigma \nu^{\frac13}t\mathcal{M})^{\theta-1} f\|_{L^1 L^2})\\
\leq &C(\|\tilde{f}(t_0)\|_{L^2}+\|(1+\sigma \nu^{\frac13}t\mathcal{M})^{\theta} g\|_{L^1 L^2}+\sigma\theta\nu^{\frac13}\| \mathcal{M}\tilde{f}\|_{L^1 L^2}).
\end{align*}
Choosing $C \sigma \theta\leq \frac12$, we obtain
\begin{align*}
\|\tilde{f}\|_{X}+ \nu^{\frac13}\|\mathcal{M} \tilde{f}\|_{L^1 L^2}\leq C\|\tilde{f}(t_0)\|_{L^2}+C\|(1+\sigma \nu^{\frac13}t\mathcal{M})^{\theta} g\|_{L^1 L^2}.
\end{align*}
Therefore, we have
\begin{align*}
&\|f\|_{X_{\theta}}+\nu^{\frac13}\|\mathcal{M}(1+ \nu^{\frac13}t\mathcal{M})^{\theta}f\|_{L^1L^2}
\leq \sigma^{-\theta}(\|\tilde{f}\|_{X}+\nu^{\frac13}\|\mathcal{M} \tilde{f}\|_{L^1 L^2})\\
\lesssim & \|(1+\nu^{\frac13} t_0 \mathcal{M})^{\theta}f(t_0)\|_{L^2}+\|(1+ \nu^{\frac13}t\mathcal{M})^{\theta} g\|_{L^1 L^2},
\end{align*}
which completes the proof of \eqref{fX large time} and \eqref{esti: nu1/3M f}.
\end{proof}
\section{The proof of Theorem \ref{Th1}}\label{sec:nonlinear}
In this section, we provide the estimates for the solution to the linear equation \eqref{linear} in Subsection \ref{sec:linear esti} and error equation \eqref{omerr} in Subsection \ref{sec: error}, respectively. Finally, we prove Theorem \ref{Th1} in Subsection \ref{sec:proof th1}.
\subsection{The estimates for the linear equation}\label{sec:linear esti}
Consider
\begin{equation}\label{linear}
\left\{
\begin{aligned}
&\pa_t\om_{L}+\nu(1-t^2\pa^2_x)\om_{L}+y\pa_x\om_{L}=0,\\
&\om_{L}\big|_{t=0}=\om^{in}(x,y),\quad \om_{L}(\pm1)=0.
\end{aligned}
\right.
\end{equation}
We have the following enhanced dissipation and inviscid damping estimates.
\begin{proposition}\label{Pro: linear}
Let $\om_{L}(t,x,y)$ be the solution of \eqref{linear} and $u_{L}=(\pa_y, -\pa_x)\De^{-1} \om_{L}$. For $s\geq 0$, there holds the weak enhanced dissipation estimates
\begin{align}
&\|(1+ \nu^{\frac13}t\mathcal{M})^{s}\om_{L}\|_{L^2}+ \|(1+\nu^{\frac13}t\mathcal{M})^{s}(\pa_y+t\pa_x)\om_{L}\|_{L^2}\leq C_s (1+\nu t^3)^{-\frac14}e^{-\nu t}E_0, \label{paxpayomL L2}\\
&\|(1+ \nu^{\frac13}t\mathcal{M})^{s}\pa_x\om_{L}\|_{L^2}\leq C_s (1+\nu t^3)^{-\frac34}e^{-\nu t}E_0,\label{paxomL2}\\
&\|(1+ \nu^{\frac13}t\mathcal{M})^{s}\pa_x\om_{L}\|_{L^\infty}\leq C_s (1+\nu t^3)^{-1}e^{-\nu t}E_0, \label{omL Linfty}\\
&\|(1+ \nu^{\frac13}t\mathcal{M})^{s}(\pa_y+t\pa_x)\om_{L}\|_{L^\infty}\leq C_s (1+\nu t^3)^{-\frac12}e^{-\nu t}E_0, \label{paxomL Linfty}\\
&\|(1+ \nu^{\frac13}t\mathcal{M})^{s}\pa^2_x\om_{L}\|_{L^\infty}\leq C_s (\nu t^3)^{-\frac13}(1+\nu t^3)^{-\frac76}e^{-\nu t}E_0, \label{paxxxomL Linfty}
\end{align}
and the inviscid damping estimates
\begin{align}
&\|(1+ \nu^{\frac13}t\mathcal{M})^{s}u^{(1)}_L\|_{L^\infty}+\|(1+ \nu^{\frac13}t\mathcal{M})^{s}(u^{(1)}_L-tu^{(2)}_L)\|_{L^\infty}\leq C_s (1+t)^{-1}e^{-\nu t}E_0,\label{u1LLinfty}\\
&\|(1+ \nu^{\frac13}t\mathcal{M})^{s}u^{(2)}_{L}\|_{L^\infty}\leq C_s(1+t)^{-2}\ln(1+t)e^{-\nu t}E_0,\label{u12L22}
\end{align}
and the error estimate
\begin{equation} \label{ErL2}
\begin{aligned}
&\|(1+ \nu^{\frac13}t\mathcal{M})^{s}E_r\|_{L^2}\\
\leq &C_s \Big((1+t)^{-1}(1+\nu t^3)^{-\frac14} E_0 +\nu(1+\nu t^3)^{-\frac14}+\nu^{\frac23}(1+\nu t^3)^{-\frac{5}{12}}\Big)e^{-\nu t}E_0.
\end{aligned}
\end{equation}
\end{proposition}
\begin{proof}
Let $w^L_k= \int_{\R} e^{-ikx} \om_{L}dx$ be the solution of \eqref{linear Fourier}, satisfying the estimates in Lemma \ref{lemma:wL estimate} with $A_j=\|(\pa_y,k)^jw^{in}_{k}\|_{L^2_y}$.
We then divide the proof into three steps.
In the following, we repeatedly use
\begin{align*}
&\int_{|k|\leq 1}|k|^m e^{-c\nu |k|^{2}t^3}dk\leq \min\{1, (\nu t^3)^{-\frac{m+1}{2}}\}\leq C(1+\nu t^3)^{-\frac{m+1}{2}},\quad m\geq 0.
\end{align*}

\noindent \textbf{Step 1.} The proof of \eqref{paxpayomL L2}--\eqref{paxxxomL Linfty}.
It is easy to get that
\begin{align*}
&\om_{L}=\int_{\R} w^{L}_{k} e^{ikx}dk, \quad \pa_x\om_{L}=\int_{\R} ikw^{L}_k e^{ikx} dk =\int_{\R} ik(w^L_k e^{ikyt})e^{ik(x-y t)}dk,\\
&(\pa_y+t\pa_x)\om_{L}=\int_{\R}(\pa_y+ikt)w^L_k e^{ikx}dk=\int_{\R}\pa_y(e^{ikyt}w^L_{k}) e^{ik(x-yt)}dk.
\end{align*}
By \eqref{kneq0 wL2} and $\mathcal{M}(k)\leq 3|k|^{\frac23}$, we obtain
\begin{align*}
 &\|(1+ \nu^{\frac13}t\mathcal{M})^{s}\om_{L}\|^2_{L^2}+\|(1+ \nu^{\frac13}t\mathcal{M})^{s}(\pa_y+t\pa_x)\om_{L}\|^2_{L^2}\\
 \leq & 3\int_{\R}(1+ \nu^{\frac13}|k|^{\frac23} t)^{2s} \|(\pa_y,1)(e^{ikyt}w^L_k)\|^2_{L^2_y} dk\\
\leq & C_s\int_{|k|\leq 1} e^{-2\nu t-\frac13\nu k^2 t^3}|A_1(k)|^2 dk+ \int_{|k|\geq 1} e^{-2\nu t-\frac13\nu k^2 t^3}|A_1(k)|^2 dk\\
\leq &C_s(1+\nu t^3)^{-\frac12}e^{-2\nu t} (\|A_1(k)\|^2_{L^\infty_k(|k|\leq 1)}+\|A_1(k)\|^2_{L^2_k(|k|\geq 1)})\leq C_s(1+\nu t^3)^{-\frac12} e^{-2\nu t} E^2_0.
\end{align*}
Similarly, we also have
\begin{align*}
\|(1+ \nu^{\frac13}t\mathcal{M})^{s}\pa_x\om_{L}\|_{L^2}
\leq &C_s(1+\nu t^3)^{-\frac34}e^{-\nu t} E_0.
\end{align*}

Next, we consider the estimates in $L^\infty$ setting. It follows from \eqref{kneq0 wLinfty} that
\begin{align*}
\|(1+\nu^{\frac13}t\mathcal{M})^{s}\pa_x\om_{L}\|_{L^\infty}
\leq &\int_{\R} (1+ \nu^{\frac13}t\mathcal{M})^{s}|k|\| (e^{ik yt} w^L_k)\|_{L^\infty_y} d k\\
\leq& C_s\int_{\R} |k|e^{-\nu t-\frac16\nu k^2 t^3} A_1(k) d k
\leq  C_s (1+\nu t^{3})^{-1}e^{-\nu t} E_0.
\end{align*}
In a similar way, there holds
\begin{align*}
\|(1+\nu^{\frac13}t\mathcal{M})^{s}(\pa_y+t\pa_x)\om_{L}\|_{L^\infty}
\leq &C_s (1+\nu t^3)^{-\frac12}e^{-\nu t} E_0.
\end{align*}

For \eqref{paxxxomL Linfty},  thanks to \eqref{kneq0 wLinfty}, we have
\begin{equation}\label{paxxomLinfty}
\begin{aligned}
&\|(1+ \nu^{\frac13}t\mathcal{M})^{s}\pa^2_x\om_{L}\|_{L^\infty}\leq C_s\int_{\R} \|k^2 (e^{ik yt} w^L_k)\|_{L^\infty_y} d k\\
\leq &C_se^{-\nu t}\Big(\int_{|k|\leq 1}|k|^2e^{-\frac13 \nu k^2 t^3}A_1 dk+\int_{|k|\geq 1} |k|^{2}e^{-\frac13 \nu k^2 t^3}|k|^{-\frac12}A_1dk\Big).
\end{aligned}
\end{equation}
For $|k|\leq 1$, we use $|k|^{\frac23}e^{-\frac13\nu k^2 t^3}\leq C(\nu t^3)^{-\frac13}$ and $|k|^2e^{-\frac16\nu |k|^2 t^3}\leq C(\nu t^3)^{-1}$ to obtain
\begin{equation}\label{paxxom1}
\begin{aligned}
\int_{|k|\leq 1}|k|^2e^{-\frac13 \nu k^2 t^3} A_1dk\leq &C(\nu t^3)^{-\frac13}\|A_1\|_{L^2_k},\\
\int_{|k|\leq 1}|k|^2e^{-\frac13 \nu k^2 t^3}A_1 dk\leq & C (\nu t^3)^{-1} \int_{|k|\leq 1}e^{-\frac16 \nu k^2 t^3}A_1 dk\leq C(\nu t^3)^{-\frac32}\|A_1\|_{L^\infty_k}.
\end{aligned}
\end{equation}
For $|k|\geq 1$, due to
$$ e^{-\frac13 \nu |k|^2 t^3}\leq C(\nu k^2 t^3)^{-\frac13},\quad  |k|^2e^{-\frac13 \nu |k|^2 t^3}\leq C(\nu t^3)^{-1}e^{-\frac16 \nu k^2 t^3}\leq C(\nu t^3)^{-\frac32},$$
we deduce
\begin{equation}\label{paxxom2}
\begin{aligned}
\int_{|k|\geq 1}|k|^2e^{-\frac13\nu |k|^2 t^3}|k|^{-\frac12} A_1 dk\leq& C \int_{|k|\geq 1} (\nu k^2 t^3)^{-\frac13}|k|^{-\frac12}A_3 dk\leq C(\nu t^3)^{-\frac13}\|A_3\|_{L^2_k},\\
\int_{|k|\geq 1}|k|^2e^{-\frac13\nu |k|^2 t^3} |k|^{-\frac12}A_1 dk
\leq&  C (\nu t^3)^{-\frac32} \|A_3\|_{L^2_k}.
\end{aligned}
\end{equation}
Inserting \eqref{paxxom1} and \eqref{paxxom2} into \eqref{paxxomLinfty}, we get \eqref{paxxxomL Linfty}.

\textbf{Step 2.} The proof of \eqref{u1LLinfty} and \eqref{u12L22}.

Let $\widetilde{\om}_{L}(t,x,y)=\om_{L}(t,x+ty,y)$. It follows from \eqref{linear} that
\begin{align*}
(\pa_t+\nu(1-t^2\pa^2_x))\widetilde{\om}_{L}=0, \quad \widetilde{\om}_{L}\big|_{t=0}=\om^{in}(x,y), \quad \widetilde{\om}_{L}(t,x\mp t,\pm1)=0.
\end{align*}
The solution can be written as
\begin{align*}
\widetilde{\om}_{L}=e^{-\nu t+\frac13\nu t^3\pa^2_x} \om^{in}.
\end{align*}
Using the heat kernel estimate, we obtain
\begin{align}\label{omLxy}
\|(1+ \nu^{\frac13}t\mathcal{M})^{s}\widetilde{\om}_{L}\|_{L^1_xL^\infty_y}
\leq C_s e^{-\nu t}\|\om^{in}\|_{L^1_x L^\infty_y}.
\end{align}
Let $(1+ \nu^{\frac13}t\mathcal{M})^{s}\phi_1$ be the solution of $\pa^2_y (1+\nu^{\frac13}t\mathcal{M})^{s} \phi_1=(1+ \nu^{\frac13}t\mathcal{M})^{s}\om_{L}$ with $\phi_1(\pm1)=0.$ By solving the ODE, we obtain
\begin{align*}
(1+ \nu^{\frac13}t\mathcal{M})^{s}\phi_1(y)
=&-\frac12\int^y_{-1}(1-y)(1+y') (1+ \nu^{\frac13}t\mathcal{M})^{s}\om_{L} dy'\\
&-\frac{1}{2}\int^1_y (1+y)(1-y')(1+ \nu^{\frac13}t\mathcal{M})^{s}\om_{L} dy'.
\end{align*}
It follows that
\begin{equation}\label{phi1Linfty}
\begin{aligned}
&\|\pa_y(1+ \nu^{\frac13}t\mathcal{M})^{s}\phi_1\|_{L^\infty_y}\\
\leq& \|(1+ \nu^{\frac13}t\mathcal{M})^{s}\om_{L}\|_{L^1_y}= \int^{1}_{-1}|(1+ \nu^{\frac13}t\mathcal{M})^{s}(\tilde{\om}_{L}(t,x-ty,y))|dy\\
\leq&C_st^{-1}\int_{\R}|(1+\nu^{\frac13}t\mathcal{M})^{s}\tilde{\om}_{L}(t,m,y)|dm
 \leq C_st^{-1}e^{-\nu t}\|\om^{in}\|_{L^1_x L^\infty_y},
\end{aligned}
\end{equation}
where in the last inequality we used \eqref{omLxy}.

Due to $\De \phi_L=\om_{L}$, we get
\begin{align*}
\pa^2_y (1+ \nu^{\frac13}t\mathcal{M})^{s}(\phi_1-\phi_{L})=(1+ \nu^{\frac13}t\mathcal{M})^{s}\pa^2_x\phi_{L}.
\end{align*}
Let $\psi_{L}=\int_{\R}e^{-ikx} \phi_Ldx$. Using \eqref{k>1, psiL2}, \eqref{esti4 psiL2}, \eqref{kneq0 wL2} and $\mathcal{M}(k)\leq 3|k|^{\frac23}$, we obtain
\begin{equation*}
\begin{aligned}
&\|\pa_y(1+ \nu^{\frac13}t\mathcal{M})^{s}(\phi_1-\phi_{L})\|_{L^\infty_y}
\leq  \|(1+ \nu^{\frac13}t\mathcal{M})^{s}\pa^2_x\phi_{L}\|_{L^1_y}\\
\leq &3\int^1_{-1}\Big|\int_{\R} (1+ \nu^{\frac13}|k|^{\frac23} t)^{s}e^{ikx} k^2\psi_{L} dk\Big|dy\leq 6\int_{\R} (1+ \nu^{\frac13}|k|^{\frac23} t)^{s}|k|^2\|\psi_{L}\|_{L^2_y}dk\\
\leq &C_s\int_{|k|\leq 1} \frac{|k|^2e^{-\nu t}}{(1+kt)^2}A_2(k)dk+C_s\int_{|k|\geq 1}(1+t)^{-2}e^{-\nu t}|k|^{-2}A_2(k)dk\\
\leq &C_s(1+t)^{-2}e^{-\nu t}\|A_2(k)\|_{L^2_k},
\end{aligned}
\end{equation*}
which together with \eqref{phi1Linfty} implies
\begin{align}\label{payphiLinfty}
\|(1+ \nu^{\frac13}t\mathcal{M})^{s}\pa_y \phi_{L}\|_{L^\infty}\leq C_st^{-1}e^{-\nu t}(\|\om^{in}\|_{L^1_x L^\infty_y}+\|\om^{in}\|_{H^2}).
\end{align}
On the other hand, due to the interpolation inequality
\begin{align*}
\|e^{ikty}\pa_y\psi_{L}\|_{L^\infty_y}\leq \|\pa_y(e^{ikty}\pa_y\psi_{L})\|^{\frac12}_{L^2_y}\|\pa_y\psi_{L}\|^{\frac12}_{L^2_y}\leq \|\pa_y(\pa_y+ikt)\psi_{L}\|^{\frac12}_{L^2_y}\|\pa_y\psi_{L}\|^{\frac12}_{L^2_y},
\end{align*}
we use \eqref{psiL2k>1}, \eqref{paypay+kpsiL2}, \eqref{psiL2k<1}, \eqref{esti3 paypaypsi+xiytpsi} and \eqref{kneq0 wL2} to obtain
\begin{align*}
\|(1+ \nu^{\frac13}t\mathcal{M})^{s}\pa_y\phi_{L}\|_{L^\infty}\leq & 3\int_{\R} (1+ \nu^{\frac13}|k|^{\frac23} t)^{s}\|e^{ikty}\pa_y\psi_{L}\|_{L^\infty_y}dk\\
\leq & 3 \int_{\R} (1+ \nu^{\frac13}|k|^{\frac23} t)^{s}\|\pa_y(\pa_y+ikt)\psi_{L}\|^{\frac12}_{L^2_y}\|\pa_y\psi_{L}\|^{\frac12}_{L^2_y}dk\\
\leq &C_s\int_{|k|\leq 1}e^{-\nu t} A_2 dk+C_s\int_{|k|\geq 1} e^{-\nu t}|k|^{-2}A_2 dk\leq C_s e^{-\nu t}E_0.
\end{align*}
This inequality together with \eqref{payphiLinfty} implies
\begin{align}\label{payphi}
\|(1+ \nu^{\frac13}t\mathcal{M})^{s}\pa_y \phi_{L}\|_{L^\infty}\leq C_s(1+t)^{-1}e^{-\nu t}E_0.
\end{align}
In a similar way, the estimate of $\|(1+ \nu^{\frac13}t\mathcal{M})^{s}(\pa_y+t\pa_x)\phi_{L}\|_{L^\infty}$ can be obtained by taking $\pa^2_y (1+ \nu^{\frac13}t\mathcal{M})^{s} (\pa_y+t\pa_x)\phi_1=(1+ \nu^{\frac13}t\mathcal{M})^{s}(\pa_y+t\pa_x)\om_{L}$.

 By Fourier transform and interpolation inequality, there holds
\begin{align*}
\|(1+ \nu^{\frac13}t\mathcal{M})^{s}\pa_x\phi_{L}\|_{L^\infty_y}
\leq& 3\int_{\R} (1+ \nu^{\frac13}|k|^{\frac23} t)^{s}|k|\|e^{ikty}\psi_{L}\|_{L^\infty_y}dk\\
\leq& C\int_{\R} (1+ \nu^{\frac13}|k|^{\frac23} t)^{s}|k|\|(\pa_y,1)(e^{ikty}\psi_{L})\|_{L^2_y}dk.
\end{align*}
Then, using \eqref{xi>1 xi4pay+xitpsiL2}, \eqref{xi<1 pay+xitpsiL2} and \eqref{kneq0 wL2}, we deduce that
\begin{align*}
 &\int_{\R} (1+ \nu^{\frac13}|k|^{\frac23} t)^{s}|k|\|(\pa_y,1)(e^{ikty}\psi_{L})\|_{L^2_y}dk\\
\leq &C_se^{-\nu t}\Big(\int_{|k|\leq 1} \frac{|k|}{(1+kt)^2}A_3(k)dk+\int_{|k|\geq 1}(1+t)^{-2}|k|^{-3}A_3(k)dk\Big)\\
\leq &C_s(1+t)^{-2}\ln(1+t)e^{-\nu t}(\|A_3(k)\|_{L^\infty}+\|\om^{in}\|_{H^3}),
\end{align*}
which yields
\begin{align}\label{paxphiLinfty}
\|(1+ \nu^{\frac13}t\mathcal{M})^{s}\pa_x\phi_{L}\|_{L^\infty}\leq  C_s(1+t)^{-2}\ln(1+t)e^{-\nu t}E_0.
\end{align}
\textbf{Step 3.} The proof of \eqref{ErL2}.
Recall that
\begin{align*}
E_{r}=E_{r_{L}}+u_{L}\cdot\na \om_{L}=\pa_t\om_{L}-\nu \De\om_{L}+y\pa_x\om_{L}+u_{L}\cdot\na \om_{L}.
\end{align*}
Noticing that
\begin{align*}
u_{L}\cdot\na\om_{L}=u^{(1)}_{L}\pa_x\om_{L}+u^{(2)}_{L}\pa_y\om_{L}=(u^{(1)}_{L}-tu^{(2)}_{L})\pa_x\om_{L}+u^{(2)}_{L}(\pa_y+t\pa_x)\om_{L},
\end{align*}
we integrate in space and use \eqref{paxpayomL L2}, \eqref{paxomL2}, \eqref{u1LLinfty}, \eqref{u12L22} and Lemma \ref{lemma: fg esti}
to obtain
\begin{equation}\label{v1-tv2wlL2}
\begin{aligned}
&\|(1+\nu^{\frac13}t\mathcal{M})^{s}\big((u^{(1)}_{L}-tu^{(2)}_{L})\pa_x\om_{L}+u^{(2)}_{L}(\pa_y+t\pa_x)\om_{L}\big)\|_{L^2}\\
\leq & \|(1+\nu^{\frac13}t\mathcal{M})^{s}(u^{(1)}_{L}-tu^{(2)}_{L})\|_{L^\infty} \|(1+\nu^{\frac13}t\mathcal{M})^{s}\pa_x\om_{L}\|_{L^2}\\
&+\|(1+\nu^{\frac13}t\mathcal{M})^{s}u^{(2)}_{L}\|_{L^\infty} \|(1+\nu^{\frac13}t\mathcal{M})^{s}((\pa_y+t\pa_x)\om_{L})\|_{L^2}\\
\leq  &C_s(1+t)^{-1}(1+\nu t^3)^{-\frac34}e^{-\nu t} E^2_0+C_s(1+t)^{-2}\ln(1+t)(1+\nu t^3)^{-\frac14}e^{-\nu t} E^2_0\\
\leq  &C_s (1+t)^{-1}(1+\nu t^3)^{-\frac{1}{4}}e^{-\nu t} E^2_0.
\end{aligned}
\end{equation}

Thanks to \eqref{linear} and $E_{r_{L}}=(\pa_t-\nu \De+y\pa_x)\om_{L}$, we get
\begin{equation}\label{expre: Erl}
\widehat{E}_{r_{L}}=-\nu (1+t^2k^2)w^{L}_k-\nu(\pa^2_y-k^2)w^{L}_k,
\end{equation}
which implies
\begin{align*}
e^{ikyt} \widehat{E}_{r_L}=-\nu e^{ikyt}(1+\pa^2_y-k^2+ k^2t^2)w^L_k.
\end{align*}
Due to
\begin{align*}
\pa^2_y(e^{ikyt} w^L_k)
=k^2t^2e^{ikyt}w^L_k+2ikt\pa_y (e^{ikyt} w^L_k)+e^{ikyt}\pa_y^2 w^L_k,
\end{align*}
we obtain
\begin{align}\label{expre: ErL k>nu}
e^{ikyt} \widehat{E}_{r_L}=-\nu(\pa^2_y-2ikt\pa_y- k^2+1)(e^{ikyt} w^L_k),
\end{align}
which yields
\begin{equation}\label{esti: ErLk>1}
\begin{aligned}
&\|(1+\nu^{\frac13} t\mathcal{M}(k))^{s}\widehat{E}_{r_L}\|_{L^2_{k,y}}\\
\leq &3\|\nu(1+\nu^{\frac13}|k|^{\frac23} t)^{s}((\pa^2_y, k^2)(e^{ikyt }w^L_k)+ t k \pa_y(e^{ikyt} w^L_k))\|_{L^2_{k,y}}.
\end{aligned}
\end{equation}
Using Lemma \ref{lemma:wL estimate}, we obtain
\begin{align*}
&\| \nu(1+\nu^{\frac13}|k|^{\frac23} t)^{s}((\pa_y, k)^2(e^{ikyt }w^L_k)+ t k \pa_y(e^{ikyt} w^L_k))\|_{L^2_{k,y}}\\
\leq &\nu\Big(\int_{\R}|(1+\nu^{\frac13}|k|^{\frac23} t)^{s} e^{-\nu t-\frac13\nu k^2 t^3}A_2(k)|^2+ t|(1+\nu^{\frac13}|k|^{\frac23} t)^{s} k e^{-\nu t-\frac13\nu k^2 t^3}A_2(k)|^2dk\Big)^{\frac12}\\
\leq &C_s\nu e^{-\nu t}\Big(\int_{\R}| e^{-\frac16\nu k^2 t^3}A_2(k)|^2dk\Big)^{\frac12}+C_s\nu t e^{-\nu t} \Big(\int_{\R}| k e^{-\frac16\nu k^2 t^3}A_2(k)|^2dk\Big)^{\frac12} \\
\leq & C_s\nu e^{-\nu t}\big((1+\nu t^3)^{-\frac14}+t(1+\nu t^3)^{-\frac34}\big) \big(\|A_2(k)\|_{L^\infty(|k|\leq 1)}+ \|A_2(k)\|_{L^2(|k|\geq 1)}\big).
\end{align*}
Due to the fact that $\|A_2(k)\|_{L^\infty}+\|A_2(k)\|_{L^2}\leq 2E_0\leq 2c\nu^{\frac13}$ and
$$\nu t(1+\nu t^3)^{-\frac34}= \nu^{\frac23} \nu^{\frac13}t(1+\nu t^3)^{-\frac13}(1+\nu t^3)^{-\frac{5}{12}}\leq \nu^{\frac23}(1+\nu t^3)^{-\frac{5}{12}},$$
there holds
\begin{equation}\label{esti: ErL fina1}
\begin{aligned}
&\| \nu(1+\nu^{\frac13}|k|^{\frac23} t)^{s}((\pa_y, k)^2(e^{ikyt }w^L_k)+ t k \pa_y(e^{ikyt} w^L_k))\|_{L^2_{k,y}}\\
\leq &C_s (\nu(1+\nu t^3)^{-\frac14}+\nu^{\frac23}(1+\nu t^3)^{-\frac{5}{12}})e^{-\nu t}E_0.
\end{aligned}
\end{equation}
Inserting \eqref{esti: ErL fina1} into \eqref{esti: ErLk>1} and combining with \eqref{v1-tv2wlL2}, we finally obtain \eqref{ErL2}.
\end{proof}
\subsection{The estimate for $\om_{e}$}\label{sec: error}
 Recall the equation
\begin{equation}\label{omerr}
\left\{
\begin{aligned}
&\pat\om_{e}-\nu \De\om_{e}+y\pa_x\om_e+u\cdot\na\om_{e}+u_{e}\cdot\na\om_{L}+E_r=0,\\
&u_{e}=(\pa_y,-\pa_x)\phi_e, \,\phi_e=\De^{-1}\om_e,\\
&\om_{e}(t,x,\pm1)=0,\, \om_{e}\big|_{t=0}=0,
\end{aligned}
\right.
\end{equation}
where
\begin{align*}
E_{r}=E_{r_{L}}+u_{L}\cdot{\na}\om_{L}, \quad E_{r_{L}}=\pa_t \om_{L}-\nu\De \om_{L}+y\pa_{x}\om_{L}.
\end{align*}
In this subsection, we focus on the estimate $\|\om_e\|_{L^2}$. More precisely,
\begin{proposition}\label{Prop: esti ome}
Let $\om_e$ be the solution to \eqref{omerr}.
If the initial vorticity $\om^{in}$ satisfies the same condition as in Theorem \ref{Th1}, then for $\theta \in(0,\frac{1}{16})$ independent of $\nu$, there holds
\begin{align*}
\|(1+ \nu^{\frac13}t\mathcal{M})^{\theta}\om_e(t)\|_{L^2}\leq CE_0.
\end{align*}
\end{proposition}
We divide the proof into two parts by the time $[0,\nu^{-\frac16}]$ and $[\nu^{-\frac16},\infty)$ and establish the estimates in Proposition \ref{small time} and  Proposition \ref{ome estimate}, respectively.
\subsubsection{The estimate of $\om_e$ for $t\leq \nu^{-\frac16}$}\label{subsec: short time}
We shall use the energy method to derive the estimate $\|(1+ \nu^{\frac13}t\mathcal{M})^{\theta}\om_e(t)\|_{L^2}$ for $t\in [0,\nu^{-\frac16}]$.
\begin{proposition}\label{small time}
Let $\om_e$ be the solution of \eqref{omerr}. For $t\in [0,\nu^{-1/6}]$, there holds
\begin{align*}
\|(1+ \nu^{\frac13}t\mathcal{M})^{\theta}\om_e(t)\|_{L^2}\leq C\nu^{1/3}\ln (1+\nu^{-1})E_0.
\end{align*}
\end{proposition}
\begin{proof}
Noticing that for $t\in[0,\nu^{-\frac16}]$, we have
\begin{align*}
\|(1+ \nu^{\frac13}t\mathcal{M})^{\theta}\om_e(t)\|_{L^2_{x,y}}\leq 3\|(1+ \nu^{\frac13} t)^{\theta}\widehat{\om}_e(t)\|_{L^2_{k,y}}\leq 6\|\om_{e}\|_{L^2}.
\end{align*}
To conclude the proof, it remains to bound $\|\om_{e}\|_{L^2}$.
Thanks to $\om_e|_{y=\pm1}=0$, $\na \cdot u=0$, we take the inner product of \eqref{omerr} with $\om_e$ to obtain
\begin{align*}
\frac12\frac{d}{dt}\|\om_e\|^2_{L^2}+\nu\|\na \om_e\|^2_{L^2}=\lan -(u_e\cdot\na \om_{L})-E_r, \om_e \ran\leq (\|u_e\cdot\na \om_{L}\|_{L^2}+\|E_r\|_{L^2})\|\om_e\|_{L^2},
\end{align*}
which implies
\begin{align*}
\frac{d}{dt}\|\om_e\|_{L^2}\leq \|u_e\|_{L^2}\|\na \om_L\|_{L^\infty}+\|E_r\|_{L^2}.
\end{align*}
Thanks to $\|u_e\|_{L^2}\leq \|\om_e\|_{L^2}$ and Gronwall's inequality, we have
\begin{align}\label{omeL2 small}
\|\om_e\|_{L^2}\leq \exp\Big(C\int^t_0\|\na \om_{L}\|_{L^\infty}ds\Big)\int^t_0\| E_{r}(s)\|_{L^2}ds.
\end{align}
It follows from \eqref{paxpayomL L2} and \eqref{omL Linfty} that
\begin{equation}\label{naomL2}
\begin{aligned}
\int^{\nu^{-\frac16}}_0\|\na\om_{L}\|_{L^\infty}dt\leq \int^{\nu^{-\frac16}}_0(1+t)\|\pa_x \om_{L}\|_{L^\infty}+\|(\pa_y+t\pa_x)\om_{L}\|_{L^\infty}dt
\lesssim \nu^{-\frac13}E_0.
\end{aligned}
\end{equation}
Due to \eqref{ErL2}, we get
\begin{align}\label{Er small}
\int^{\nu^{-\frac16}}_0\|E_{r}\|_{L^2}
\lesssim  E_0\int^{\nu^{-\frac16}}_0(1+t)^{-1}E_0 +\nu+\nu^{\frac23}(1+\nu t^3)^{-\frac{1}{6}}dt \lesssim \nu^{\frac13}\ln (1+\nu^{-1}) E_0.
\end{align}
Inserting \eqref{naomL2} and \eqref{Er small} into \eqref{omeL2 small}, we finally obtain
\begin{align*}
\|\om_e(t)\|_{L^2}\leq C\exp(C\nu^{-\frac13}E_0)\nu^{\frac13}\ln (1+\nu^{-1})E_0\leq C\nu^{1/3}\ln (1+\nu^{-1}) E_0.
\end{align*}
Thus, we complete the proof of Proposition \ref{small time}.
\end{proof}

\subsubsection{The estimate of $\om_e$ for $t\geq \nu^{-\frac16}$}\label{subsec: long time}
In this subsection, we study the space-time estimate for $\om_e$ with $t\geq \nu^{-\frac16}$.
Recall
\begin{align*}
\|f\|_{X_{\theta}}=&\|(1+\nu^{\frac13}t\mathcal{M})^{\theta}f\|_{L^\infty L^2}+\nu^{\frac12}\|(1+\nu^{\frac13}t\mathcal{M})^{\theta}\na f\|_{L^2 L^2}\\
&+\|(1+\nu^{\frac13}t\mathcal{M})^{\theta}\mathcal{M}_1\na \De^{-1}f\|_{L^2 L^2}+\nu^{\frac12}\|(1+\nu^{\frac13}t\mathcal{M})^{\theta}\pa_x f\|_{L^1 L^2}.
\end{align*}
\begin{proposition}\label{ome estimate}
Let $\om_e$ be the solution of \eqref{omerr}. For $t\geq \nu^{-1/6}$, there holds
\begin{align*}
\|\om_e(t)\|_{X_{\theta}}\leq CE_0.
\end{align*}
\end{proposition}
As introduced in Section \ref{sec:Introduction}, we decompose $\om_e$ into $\om_{1}+\om_{2}$ and establish the estimates for $\om_1$ and $\om_2$ in Proposition \ref{lemma: g} and Proposition \ref{pro: wre}, respectively.

\vspace{0.2cm}

\begin{proposition}\label{lemma: g}
Let $\om_{1}$ be the solution of
\begin{equation}
\left\{
\begin{aligned}\label{equ: om*1}
&\pa_t\om_{1}-\nu\De\om_{1}+y\pa_x\om_{1}+u\cdot\na\om_e+u^{(1)}_e\pa_x\om_{L}+u^{(2)}_e(\pa_y+t\pa_x)\om_{L}+E_r=0, \\
&\om_{1}(t,x,\pm1)=0, \quad \om_{1}|_{t=T_{0}}=\om_{e}(T_{0}).
\end{aligned}
\right.
\end{equation}
There exists a small constant $\epsilon$ such that for $j_0=\log_2 \nu^{-\frac23-\epsilon}$, we have
\begin{align}
&\|\om_{1}\|_{X(T_0,T_{j_0})}\leq C \nu^{-\frac13}E_0\|\om_e\|_{X(T_0,T_{j_0})}+C\nu^{-\frac58} \|\om_e\|^2_{X(T_0,T_{j_0})} +C\nu^{\frac13-\epsilon}E_0, \label{esti om1 tildeX}\\
&\|\om_{1}\|_{X_{\theta}(T_{j_0},T)}\leq C\nu^{-\frac58}\|\om_e\|^2_{X_{\theta}(T_{j_0},T)}+C\nu^{\frac13}E_0+C\|(1+ \nu^{\frac13}T_{j_0} \mathcal{M})^{\theta}\om_{1}(T_{j_0})\|_{L^2}.\label{esti om1 X large}
\end{align}
\end{proposition}
\begin{proof}
Taking
$$g\triangleq u\cdot\na\om_e+u^{(1)}_e\pa_x\om_{L}+u^{(2)}_e(\pa_y+t\pa_x)\om_{L}+E_r,$$
we use \eqref{esti: ftildeX} with $t_0=T_0$ and Proposition \ref{small time} to obtain that
\begin{align*}
\|\om_{1}\|_{X(T_0,T_{j_0})}\lesssim \|\om_{1}(T_0)\|_{L^2}+\|g\|_{L^1(T_0,T_{j_0};L^2)}\lesssim  \nu^{\frac13-\epsilon}E_0+\|g\|_{L^1(T_0,T_{j_0};L^2)},
\end{align*}
where we use the fact that $\om_{1}(T_0)=\om_{e}(T_0)$.

By \eqref{fX large time} with $t_0=T_{j_0}=2^{j_0}\nu^{-\frac13}$, we have
\begin{align*}
\|\om_{1}\|_{X_{\theta}(T_{j_0},T)}\lesssim \|(1+ \nu^{\frac13}T_{j_0} \mathcal{M})^{\theta}\om_{1}(T_{j_0})\|_{L^2}+ \|(1+\nu^{\frac13}t\mathcal{M})^{\theta}g\|_{L^1(T_{j_0},T; L^2)}.
\end{align*}
It suffices to prove
\begin{align}
\|g\|_{L^1(T_0,T_{j_0};L^2)}\lesssim & \nu^{-\frac13}E_0\|\om_e\|_{X(T_0,T_{j_0})}+\nu^{-\frac58} \|\om_e\|^2_{X(T_0,T_{j_0})} +\nu^{\frac13-\epsilon}E_0,\label{esti: g small}\\
\|(1+\nu^{\frac13}t\mathcal{M})^{\theta}g\|_{L^1 (T_{j_0},T; L^2) }\lesssim & \nu^{-\frac58} \|\om_e\|^2_{X_{\theta}(T_0,T_{j_0})}+\nu^{\frac13}E_0.\label{esti: g}
\end{align}
The proof of \eqref{esti: g} appears to be more complex than that of \eqref{esti: g small}. Here, we primarily present the detailed argument of \eqref{esti: g}, which is divided into three steps. While we highlight the key differences between the proofs of \eqref{esti: g} and \eqref{esti: g small}.

\vspace{0.2cm}

\textbf{Step 1.} The estimate of $u^{(1)}_e\pa_x\om_{L}+u^{(2)}_e(\pa_y+t\pa_x)\om_{L}$.
Thanks to Proposition \ref{Pro: linear}, Lemma \ref{lemma: fg esti} and the fact that $\|u_e\|_{L^2}\leq \|\om_e\|_{L^2}$, we obtain
\begin{equation*}
\begin{aligned}
&\|(1+\nu^{\frac13}t\mathcal{M})^{\theta}(u^{(1)}_e\pa_x\om_{L}+u^{(2)}_e(\pa_y+t\pa_x)\om_{L})\|_{ L^2 } \\
\leq&\|(1+\nu^{\frac13}t\mathcal{M})^{\theta}u_e\|_{L^2}(\|(1+\nu^{\frac13}t\mathcal{M})^{\theta}\pa_x\om_{L}\|_{L^\infty} +\|(1+\nu^{\frac13}t\mathcal{M})^{\theta}(\pa_y+t\pa_x)\om_{L}\|_{L^\infty}).
\end{aligned}
\end{equation*}
Integrating in time, we deduce
\begin{equation}\label{weight-u1paxomL}
\begin{aligned}
&\|(1+\nu^{\frac13}t\mathcal{M})^{\theta}(u^{(1)}_e\pa_x\om_{L}+u^{(2)}_e(\pa_y+t\pa_x)\om_{L})\|_{L^1(T_{j_0},T; L^2) } \\
\leq& \|(1+\nu^{\frac13}t\mathcal{M})^{\theta}\om_{e}\|_{L^\infty(T_{j_0},T; L^2)}\|(1+\nu t^3)^{-\frac12}\|_{L^1(T_{j_0},T)}E_0\\
\leq& CE_0\|\om_{e}\|_{X_{\theta}(T_{j_0},T)}.
\end{aligned}
\end{equation}

For $t\in (T_0, T_{j_0})$, there holds
\begin{align*}
\|(1+\nu t^3)^{-\frac12}\|_{L^1(T_0,T_{j_0})}\leq C\nu^{-\frac13}.
\end{align*}
Then, a similar calculation as \eqref{weight-u1paxomL} shows that
\begin{align}\label{u1epaxom+u2payom}
\|u^{(1)}_e\pa_x\om_{L}+u^{(2)}_e(\pa_y+t\pa_x)\om_{L}\|_{L^1(T_{0},T_{j_0}; L^2)}\leq C\nu^{-\frac13}E_0\|\om_e\|_{X(T_{0},T_{j_0})}.
\end{align}

\textbf{Step 2.} The estimate of $u\cdot\na \om_e$.
We first decompose it into
$$u\cdot\na \om_e=u_e\cdot\na \om_e+u_L\cdot\na \om_e.$$

For $u_L\cdot\na \om_e$, we use Lemma \ref{lemma: fg esti} to obtain
\begin{equation}
\begin{aligned}\label{uLomeL1L2}
&\|(1+\nu^{\frac13}t\mathcal{M})^{\theta}(u_L\cdot\na \om_e)\|_{ L^1(T_{j_0},T; L^2)}\\
\leq &\|(1+\nu^{\frac13}t\mathcal{M})^{\theta}(u_{L}^{(1)}\pa_x \om_e)\|_{ L^1(T_{j_0},T; L^2)}+\|(1+\nu^{\frac13}t\mathcal{M})^{\theta}(u_{L}^{(2)}\pa_y \om_e)\|_{L^1(T_{j_0},T; L^2)}\\
\leq &\|(1+\nu^{\frac13}t\mathcal{M})^{\theta}u^{(1)}_L\|_{L^\infty(T_{j_0},T; L^\infty)}\|(1+\nu^{\frac13}t\mathcal{M})^{\theta}\pa_x \om_e\|_{L^1(T_{j_0},T; L^2)}\\
&+\|(1+\nu^{\frac13}t\mathcal{M})^{\theta}u^{(2)}_L\|_{L^2(T_{j_0},T;L^\infty)}\|(1+\nu^{\frac13}t\mathcal{M})^{\theta}\pa_y \om_e\|_{L^2(T_{j_0},T; L^2)}.
\end{aligned}
\end{equation}
By \eqref{u1LLinfty} and \eqref{u12L22}, we deduce
\begin{equation*}
\begin{aligned}
&\|(1+\nu^{\frac13}t\mathcal{M})^{\theta}u^{(1)}_L\|_{L^\infty(T_{j_0},T; L^\infty)}\leq T^{-1}_{j_0}E_0\leq \nu E_0,\\
&\|(1+\nu^{\frac13}t\mathcal{M})^{\theta}u^{(2)}_L\|_{L^2(T_{j_0},T;L^\infty)}\leq \|(1+t)^{-2}\ln(1+t)\|_{L^2(T_{j_0}, T)} E_0\leq \nu^{\frac32}E_0.
\end{aligned}
\end{equation*}
Inserting above two estimates into \eqref{uLomeL1L2} and using the definition of $X$-norm \eqref{fXk}, we arrive at
\begin{align}\label{weight-uLnaome}
\|(1+\nu^{\frac13}t\mathcal{M})^{\theta}(u_L\cdot\na \om_e)\|_{ L^1(T_{j_0},T; L^2)}
\lesssim \nu^{\frac12}\|\om_{e}\|_{X_{\theta}(T_{j_0},T)}E_0.
\end{align}

For $(T_0,T_{j_0})$, there holds
\begin{equation*}
\begin{aligned}
&\|u^{(1)}_L\|_{L^\infty(T_{j_0},T; L^\infty)}\leq T^{-1}_{0}E_0\leq \nu^{\frac16} E_0,\\
&\|u^{(2)}_L\|_{L^2(T_{j_0},T;L^\infty)}\leq \|(1+t)^{-2}\ln(1+t)\|_{L^2(T_{0}, T_{j_0})} E_0\leq \nu^{\frac16}E_0,
\end{aligned}
\end{equation*}
which implies
\begin{align}\label{uLnaome}
\|u_L\cdot\na \om_e\|_{ L^1(T_{0},T_{j_0}; L^2)}
\lesssim \nu^{-\frac13}\|\om_{e}\|_{X(T_{0},T_{j_0})}E_0.
\end{align}

Next, we get into the estimate of  $u_e\cdot\na \om_e$. By Lemma \ref{lemma: fg esti}, we have
\begin{align*}
&\|(1+\nu^{\frac13}t\mathcal{M})^{\theta}(u_e\cdot\na \om_e)\|_{L^1(T_{j_0},T; L^2)}\\
\leq &\|(1+\nu^{\frac13}t\mathcal{M})^{\theta}(u^{(1)}_e \pa_x\om_e)\|_{L^1(T_{j_0},T; L^2)}+\|(1+\nu^{\frac13}t\mathcal{M})^{\theta}(u^{(2)}_e\pa_y \om_e)\|_{L^1(T_{j_0},T; L^2)}\\
\leq &\big\|\|(1+\nu^{\frac13}t\mathcal{M})^{\theta}u^{(1)}_e\|_{L^\infty} \|(1+\nu^{\frac13}t\mathcal{M})^{\theta}\pa_x\om_e\|_{L^2}\big\|_{L^1(T_{j_0},T)}\\
&+\|(1+\nu^{\frac13}t\mathcal{M})^{\theta}u^{(2)}_e\|_{L^2(T_{j_0},T; L^\infty)}\|(1+\nu^{\frac13}t\mathcal{M})^{\theta}\pa_y \om_e\|_{L^2(T_{j_0},T;L^2)}
:=\textrm{I}+\textrm{II}.
\end{align*}
Define $v_{e,k}=\int_{\R} e^{-ixk}u_edx$ and $w_{e,k}=\int_{\R} e^{-ixk}\om_edx$. There holds
\begin{align*}
\|(1+\nu^{\frac13}t\mathcal{M})^{\theta}u^{(1)}_e\|_{L^\infty}\leq \int_{\R}(1+\nu^{\frac13}t\mathcal{M}(k) )^{\theta}\|v^{(1)}_{e,k}\|_{L^\infty_y} dk.
\end{align*}
For $|k|\leq 1$, using the fact $\|v^{(1)}_{e,k}\|_{L^\infty_y}\leq \|w_{e,k}\|_{L^2_y}$ and H\"older's inequality, we obtain
\begin{equation}\label{weL2l<1}
\begin{aligned}
\int_{|k|\leq 1}(1+\nu^{\frac13} t\mathcal{M}(k))^{\theta}\|v^{(1)}_{e,k}\|_{L^\infty_y} dk\leq \|(1+\nu^{\frac13}t\mathcal{M})^{\theta}\om_e\|_{L^2_x L^2_y}.
\end{aligned}
\end{equation}
For $|k|\geq 1$, by means of the interpolation
\begin{align}\label{interpo vLinfty}
\|v^{(1)}_{e,k}\|_{L^\infty_y}\leq \|v^{(1)}_{e,k}\|^{\frac34}_{L^2_y}\|\pa_y^2v^{(1)}_{e,k}\|^{\frac14}_{L^2_y}= |k|^{-\frac34}\|k v^{(1)}_{e,k}\|^{\frac34}_{L^2_y} \|\pa_y^2v^{(1)}_{e,k}\|^{\frac14}_{L^2_y},
\end{align}
 we derive
\begin{equation}\label{v1eweL2L1L2}
\begin{aligned}
&\int_{|k|\geq1}(1+\nu^{\frac13}t\mathcal{M}(k))^{\theta}\|v^{(1)}_{e,k}\|_{L^\infty_y}dk\\
\leq& \int_{|k|\geq 1}|k|^{-\frac34}(1+\nu^{\frac13}t\mathcal{M}(k))^{\theta}\|kv^{(1)}_{e,k}\|^{\frac34}_{L^2_y}\|\pa_y^2v^{(1)}_{e,k}\|^{\frac14}_{L^2_y}dk\\
\leq& \||k|^{-\frac34}\|_{L^2(|k|\geq 1)}\|(1+\nu^{\frac13}t\mathcal{M}(k))^{\theta}k v^{(1)}_{e,k}\|_{L^2_kL^2_y}^{\frac34}\|(1+\nu^{\frac13}t\mathcal{M}(k))^{\theta}\pa^2_y v^{(1)}_{e,k}\|_{L^2_kL^2_y}^{\frac14}\\
\leq & \|(1+\nu^{\frac13}t\mathcal{M}(k))^{\theta}k v^{(1)}_{e,k}\|^{\frac34}_{L^2_kL^2_y}\|(1+\nu^{\frac13}t\mathcal{M}(k))^{\theta}\pa_y w_{e,k}\|^{\frac14}_{L^2_kL^2_y}.
\end{aligned}
\end{equation}
Appealing to \eqref{weL2l<1} and \eqref{v1eweL2L1L2}, we deduce
\begin{equation}\label{esti I}
\begin{aligned}
\textrm{I}\leq& \|(1+\nu^{\frac13}t\mathcal{M})^{\theta}\om_e\|_{L^\infty(T_{j_0},T;L^2)}\|(1+\nu^{\frac13}t\mathcal{M})^{\theta}\pa_x\om_e\|_{L^1(T_{j_0},T;L^2)}\\
&+\|(1+\nu^{\frac13}t\mathcal{M})^{\theta}\mathcal{M}_1\na \De^{-1}\om_e\|^{\frac34}_{L^2(T_{j_0},T;L^2)}\|(1+\nu^{\frac13}t\mathcal{M})^{\theta}\na  \om_e\|^{\frac54}_{L^2(T_{j_0},T;L^2)}\\
\leq &C\nu^{-\frac58}\|\om_e\|^2_{X_{\theta}(T_{j_0},T)}.
\end{aligned}
\end{equation}

For $\textrm{II}$, it remains to analyze the contribution of the factor $\|(1+\nu^{\frac13}t\mathcal{M})^{\theta}u^{(2)}_e\|_{L^\infty}$. There holds
\begin{align*}
\|(1+\nu^{\frac13}t\mathcal{M})^{\theta}u^{(2)}_e\|_{L^\infty}\leq \int_{\R}(1+\nu^{\frac13}t\mathcal{M}(k))^{\theta}\|v^{(2)}_{e,k}\|_{L^\infty_y} dk.
\end{align*}
Let $\psi_e=(\pa^2_y-k^2)^{-1}w_e$. For $|k|\leq 1$, the interpolation inequality gives
\begin{align}\label{esti: ve2}
\|v^{(2)}_{e,k}\|_{L^\infty_y}=\|k\psi_e(k)\|_{L^\infty_y}\leq |k|\|\psi_e\|^{\frac34}_{L^2_y}\|\pa_y^2\psi_e\|^{\frac14}_{L^2_y}.
\end{align}
Due to $\psi_e(\pm1)=\pa^2_y\psi_e(\pm1)=0$, we obtain
\begin{align*}
&\|\psi_e\|_{L^2_y}\leq \|\pa_y\psi_e\|_{L^2_y}=\|v^{(1)}_{e,k}\|_{L^2_y},\\
&\|\pa^2_y\psi_e\|_{L^2_y}\leq \|\pa_y^3\psi_e\|_{L^2_y}=\|\pa^2_yv^{(1)}_{e,k}\|_{L^2_y}.
\end{align*}
Inserting above two estimates into \eqref{esti: ve2}, we deduce
\begin{align*}
\|v^{(2)}_{e,k}\|_{L^\infty_y}\leq |k|\|v^{(1)}_{e,k}\|^{\frac34}_{L^2_y}\|\pa_y^2v^{(1)}_{e,k}\|^{\frac14}_{L^2_y}.
\end{align*}
Thus, we have
\begin{equation}\label{ue2l<1}
\begin{aligned}
&\int_{|k|\leq1}(1+\nu^{\frac13}t\mathcal{M}(k))^{\theta}\|v^{(2)}_{e,k}\|_{L^\infty_y}dk\\
\leq &\int_{|k|\leq 1} |k|(1+\nu^{\frac13}t\mathcal{M}(k))^{\theta}\|v_{e,k}\|^{\frac34}_{L^2_y}\|\pa_y^2v_{e,k}\|^{\frac14}_{L^2_y} dk\\
\leq &\|(1+\nu^{\frac13}t\mathcal{M}(k))^{\theta} |k|^{\frac12} v_{e,k}\|^{\frac34}_{L^2_kL^2_y}\|(1+\nu^{\frac13}t\mathcal{M}(k))^{\theta}\pa_y w_{e,k}\|^{\frac14}_{L^2_kL^2_y}.
\end{aligned}
\end{equation}
For $|k|\geq 1$, the same calculations as \eqref{interpo vLinfty} and \eqref{v1eweL2L1L2} give that
\begin{equation}
\begin{aligned}\label{ue2l>1}
&\int_{|k|\geq1}(1+\nu^{\frac13}t\mathcal{M}(k))^{\theta}\|v^{(2)}_{e,k}\|_{L^\infty_y}dk\\
\leq &C\|(1+\nu^{\frac13}t\mathcal{M}(k))^{\theta}k v^{(2)}_{e,k}\|^{\frac34}_{L^2_kL^2_y}\|(1+\nu^{\frac13}t\mathcal{M}(k))^{\theta}\pa_y w_{e,k}\|^{\frac14}_{L^2_kL^2_y}.
\end{aligned}
\end{equation}
Appealing to \eqref{ue2l<1} and \eqref{ue2l>1}, we derive
\begin{align}\label{esti II}
\textrm{II}\leq C\nu^{-\frac58}\|\om_e\|^2_{X_{\theta}(T_{j_0},T)}.
\end{align}
Combining  \eqref{esti I} and \eqref{esti II}, we obtain
\begin{align}\label{weight-uenaome}
\|(1+\nu^{\frac13}t\mathcal{M})^{\theta}(u_e\cdot\na \om_e)\|_{L^1(T_{j_0},T;L^2)}\leq C\nu^{-\frac58}\|\om_e\|^2_{X_{\theta}(T_{j_0},T)}.
\end{align}

A similar calculation to $\textrm{I}$ and $\textrm{II}$ implies
\begin{align}\label{uenaome}
\|u_e\cdot\na \om_e\|_{L^1(T_0,T_{j_0};L^2)}\leq C\nu^{-\frac58}\|\om_e\|^2_{X(T_0,T_{j_0})}.
\end{align}
\textbf{Step 3.} The estimate of $E_r$.
Due to $(1+\nu t^3)^{-\frac14}\leq (\nu^{\frac13}(1+t))^{-\epsilon}$ and $(\nu(1+t))^{\frac13}e^{-\nu t}\leq C$, we have
\begin{align*}
&\|e^{-\nu t}(1+t)^{-1}(1+\nu t^3)^{-\frac14}E_0\|_{L^1(T_{j_0},T)}+\|\nu e^{-\nu t}(1+\nu t^3)^{-\frac14}\|_{L^1(T_{j_0},T)}\\
\leq  & \|\nu^{\frac13}(1+t)^{-1}(\nu^{\frac13}(1+t))^{-\epsilon}+\nu^{\frac23}(1+\nu t^3)^{-\frac14}(1+t)^{-\frac13}\|_{L^1(T_{j_0},T)}
\leq \nu^{\frac13},
\end{align*}
and
\begin{align*}
\|e^{-\nu t}(1+t)^{-1}(1+\nu t^3)^{-\frac14}E_0\|_{L^1(T_0, T_{j_0})}+\|\nu e^{-\nu t}(1+\nu t^3)^{-\frac14}\|_{L^1(T_0, T_{j_0})}\leq \nu^{\frac13-\epsilon}.
\end{align*}
Moreover,
\begin{align*}
&\int_{T_{0}}^{T} \nu^{\frac23} e^{-\nu t}(1+\nu t^3)^{-\frac{5}{12}}dt\leq \nu^{\frac13}.
\end{align*}
Then, we use \eqref{ErL2} to deduce
\begin{align}
&\|(1+\nu^{\frac13}t\mathcal{M})^{\theta}E_r\|_{ L^1(T_{j_0},T;L^2)}\leq C\nu^{\frac13}E_0,\label{weight-Er}\\
&\|E_r\|_{ L^1(T_0,T_{j_0};L^2)}\leq C\nu^{\frac13-\epsilon}E_0.\label{estim Er}
\end{align}

Summing up \eqref{weight-u1paxomL}, \eqref{weight-uLnaome}, \eqref{weight-uenaome} and \eqref{weight-Er}, we finally derive
\begin{align*}
\|(1+\nu^{\frac13}t\mathcal{M})^{\theta}g\|_{L^1(T_{j_0},T; L^2)}\leq &C(E_0\|\om_{e}\|_{X_{\theta}(T_{j_0},T)}+\nu^{-\frac58} \|\om_e\|^2_{X_{\theta}(T_0,T_{j_0})}+\nu^{\frac13}E_0)\\
\leq& C(\nu^{-\frac58} \|\om_e\|^2_{X_{\theta}(T_0,T_{j_0})}+\nu^{\frac58}E^2_0+\nu^{\frac13}E_0)\\
\leq &C(\nu^{-\frac58} \|\om_e\|^2_{X_{\theta}(T_0,T_{j_0})}+\nu^{\frac13}E_0).
\end{align*}
Combining \eqref{u1epaxom+u2payom}, \eqref{uLnaome}, \eqref{uenaome} and \eqref{estim Er}, we obtain \eqref{esti: g small} and hence complete the proof of Proposition \ref{lemma: g}.
\end{proof}

Now, we go into the analysis of $\om_{2}$ that
\begin{align}\label{equ: omre1}
\pa_t\om_{2}-\nu\De\om_{2}+y\pa_x\om_{2}-tu^{(2)}_e\pa_x\om_{L}=0, \quad \om_{2}(t,x,\pm1)=0, \quad \om_{2}|_{t=T_{0}}=0,
\end{align}
More precisely,
\begin{proposition}\label{pro: wre}
Let $\om_{2}$ be the solution of \eqref{equ: omre1}. For $j_0=\log_2 \nu^{-\frac23-\epsilon}$, there holds
\begin{align}
\|\om_{2}\|_{X(T_0,T_{j_0})}\leq &C\nu^{-\frac13}\|\om_e\|_{X(T_0,T_{j_0})}E_0,\label{main result}\\
\|\om_2\|_{X_{\theta}(T_{j_0},T)}\leq &C\nu^{-\frac13}\|\om_e\|_{X_{\theta}(T_{j_0},T)}E_0+\|(1+\nu^{\frac13}T_{j_0} \mathcal{M})^{\theta} \om_{2}(T_{j_0})\|_{L^2}.\label{om2 Xtj}
\end{align}
\end{proposition}
In order to prove Proposition \ref{pro: wre}, we further decompose $\om_2=\sum_{j\geq 1}\om_{2,j}$, where $\om_{2,j}$ solves
\begin{equation}\label{equ: om2j}
\left\{
\begin{aligned}
&\pa_t \om_{2,j}-\nu\De\om_{2,j}+y\pa_x\om_{2,j}=tu^{(2)}_e\pa_x\om_{L}\chi_{j}, \quad t>T_{j-1}, \\
& \om_{2,j}(t,x,\pm1)=0, \quad \om_{2,j}|_{t=T_{j-1}}=0,
\end{aligned}
\right.
\end{equation}
with $T_0=\nu^{-\frac16}$, $T_j=\nu^{-\frac13} 2^j$ with $j\geq 1$ and $\chi_j$ defined as \eqref{def: chij}.

For convenience, we denote
$$X(T_{j-1},T_j):=X_{j},\quad  X_{\theta}(T_{j-1},T_j):=X_{\theta,j},\quad Y(T_{j-1},T_j):=Y_{j}.$$

Next, we focus to the estimates of $\|\om_{2,j}\|_{X_{j}}$ and $\|\om_{2,j}\|_{X_{\theta,j}}$. More precisely,
\begin{proposition}\label{prop: om2j}
Let $\om_{2,j}$ be the solution of \eqref{equ: om2j}. Then, there holds
\begin{align}
\|\om_{2,j}\|_{X_{\theta,j}}\leq C2^{-\frac{j}{4}}\nu^{-\frac13}\|(1+\nu^{\frac13}t\mathcal{M})^{\theta}\mathcal{M}_1 \na\De^{-1}\om_e\|_{L^2(T_{j-1},T_j;L^2)} E_0, \quad j\geq 1. \label{om2j large}
\end{align}
In particular, for $\theta=0$, we have
\begin{align}
\|\om_{2,j}\|_{X_j}\leq C2^{-\frac{j}{4}}\nu^{-\frac13}\|\mathcal{M}_1 \na\De^{-1}\om_e\|_{L^2(T_{j-1},T_j;L^2)} E_0, \quad j\geq 1. \label{om2j small}
\end{align}
\end{proposition}

Denote $h_{\nu,t}(t)=(1+\nu^{\frac13} t)^{\theta_1}$, where $\theta_1\in(0,\frac14)$ and $\theta_1>\theta$. Let $f_1$ and $f_2$ solve
\begin{align}\label{equ: f1}
&\pa_t f_1+y \pa_xf_1=th^2_{\nu,t}u^{(2)}_e\pa_x\om_{L}\chi_j(t),
\quad f_1|_{t=T_{j-1}}=0,\\
&\pa_t f_2-\nu\De f_2+y \pa_xf_2=\nu h^{-1}_{\nu,t}\De f_1, \quad f_2(t,x,\pm1)=0, \quad f_2|_{t=T_{j-1}}=0,
\end{align}
where $f_1(t,x,\pm1)=0$ because of $\om_{L}(t,x,\pm1)=0$.
It follows from \eqref{fY esti} by taking $g=\nu h^{-1}_{\nu,t}(t)\De f_1$ that
\begin{align}\label{esti: f2Y}
\|f_2\|_{Y_j}\leq C\|h^{-1}_{\nu,t}f_1\|_{Y_j}.
\end{align}
Let $f_3=h_{\nu,t}^{-1}f_2+h_{\nu,t}^{-2}f_1$. It satisfies
\begin{equation}\label{equ: f3}
\left\{
\begin{aligned}
&\pa_t f_3-\nu \De f_3+y\pa_xf_3+h^{-2}_{\nu,t}\pa_t h_{\nu,t}f_2 +2h^{-3}_{\nu,t}\pa_t h_{\nu,t}f_1
=tu^{(2)}_e\pa_x\om_{L}\chi_j(t),\\
&f_3|_{y=\pm1}=0,\quad f_3|_{t=T_{j-1}}=0.
\end{aligned}
\right.
\end{equation}
Then $f_4=\om_{2,j}-f_3$ solves
\begin{equation}\label{equ: f4}
\left\{
\begin{aligned}
&\pa_t f_4-\nu\De f_4+y \pa_xf_4=h^{-2}_{\nu,t}\pa_t h_{\nu,t}f_2 +2h^{-3}_{\nu,t}\pa_t h_{\nu,t}f_1,\\
&f_4(t,x,\pm1)=0,\quad f_4|_{t=T_{j-1}}=0.
\end{aligned}
\right.
\end{equation}
In the following Lemmas \ref{lemma: f4 esti}--\ref{lemma: f1Y+paxf1Y}, the $L^p$ norms $(1\leq p\leq \infty)$ with respect to time are all defined on the interval $(T_{j-1},T_j]$.
\begin{lemma}\label{lemma: f4 esti}
Let $f_4$ be the solution of \eqref{equ: f4}. There holds
\begin{align}\label{esti: f4}
\|f_4\|_{X_{\theta,j}}\leq C\|h^{-1}_{\nu,t}(t)f_1\|_{Y_j}.
\end{align}
\end{lemma}
\begin{proof}
Due to $\pa_t h_{\nu,t}(t)=\theta_1\nu^{\frac13}(1+\nu^{\frac13} t)^{\theta_1-1}$, it follows from \eqref{fX large time} with $t_0=T_{j-1}$  that
\begin{align*}
\|f_4\|_{X_{\theta,j}}\lesssim & \theta_1 \nu^{\frac13}\|(1+ \nu^{\frac13} t \mathcal{M})^{\theta}(1+ \nu^{\frac13}t)^{-\theta_1-1}f_2\|_{L^1L^2} \\
&+ \theta_1 \nu^{\frac13}\| (1+ \nu^{\frac13} t \mathcal{M})^{\theta}(1+ \nu^{\frac13}t)^{-2\theta_1-1}f_1\|_{L^1L^2}.
\end{align*}
Thanks to $\theta<\theta_1$, we use H\"older's inequality and \eqref{esti: f2Y} to obtain
\begin{align*}
& \nu^{\frac13}\|(1+ \nu^{\frac13} t \mathcal{M})^{\theta}(1+ \nu^{\frac13}t)^{-\theta_1-1} f_2\|_{L^1L^2}\\
\leq &  \nu^{\frac13}\|(1+ \nu^{\frac13}t)^{\theta-\theta_1-1}\|_{L^1}\|f_2\|_{L^\infty L^2}
\leq  C \|f_2\|_{L^\infty L^2}\leq C\|h^{-1}_{\nu,t}f_1\|_{Y_j}.
\end{align*}
Moreover, we also have
\begin{align*}
& \nu^{\frac13}\|(1+ \nu^{\frac13} t \mathcal{M})^{\theta}(1+ \nu^{\frac13}t)^{-2\theta_1-1}f_1\|_{L^1L^2}\\
\leq & \nu^{\frac13}\|(1+ \nu^{\frac13}t)^{\theta-\theta_1-1}\|_{L^1}\|h_{\nu,t}^{-1}f_1\|_{L^\infty L^2}
\leq  C\|h^{-1}_{\nu,t}f_1\|_{Y_j}.
\end{align*}
Combining the above two estimates, we directly get \eqref{esti: f4}.
\end{proof}
\begin{lemma}\label{lemma: omre1}
Let $f_3$ be the solution of \eqref{equ: f3}. There holds
\begin{align*}
\|f_3\|_{X_{\theta,j}}\leq C \|h^{-1}_{\nu,t}f_1\|_{Y_j}+C2^{-j}\nu^{-\frac13}\|(1+\nu^{\frac13}t\mathcal{M})^{\theta}\mathcal{M}_1 \na\De^{-1}\om_e\|_{L^2(T_{j-1},T_j;L^2)}E_0.
\end{align*}
\end{lemma}
\begin{proof}
By the definition of $X_{\theta,j}$ and $Y_j$, we have
\begin{align}\label{f3X}
\|f_3\|_{X_{\theta,j}}\lesssim \| (1+\nu^{\frac13}t\mathcal{M})^{\theta}f_3\|_{Y_j}+\nu^{\frac12}\|(1+\nu^{\frac13}t\mathcal{M})^{\theta}\pa_x f_3\|_{L^1(T_{j-1},T_j;L^2)}.
\end{align}
It follows from $f_3=h_{\nu,t}^{-1}f_2+h_{\nu,t}^{-2}f_1$ and \eqref{esti: f2Y} that
\begin{align}\label{lanu tf3Y}
\|(1+\nu^{\frac13}t\mathcal{M})^{\theta} f_3\|_{Y_j}\leq\|(1+\nu^{\frac13}t\mathcal{M})^{\theta}h_{\nu,t}^{-1}f_2\|_{Y_j} +\|(1+\nu^{\frac13}t\mathcal{M})^{\theta}h_{\nu,t}^{-2}f_1\|_{Y_j}\lesssim \|h^{-1}_{\nu,t}f_1\|_{Y_j}.
\end{align}
Due to $f_3=\om_{2,j}-f_4$, we have
\begin{align*}
\nu^{\frac12}\|(1+\nu^{\frac13}t\mathcal{M})^{\theta}\pa_x  f_3\|_{L^1L^2}
\leq \nu^{\frac12}\|(1+\nu^{\frac13}t\mathcal{M})^{\theta}\pa_x  f_4\|_{L^1L^2}+\nu^{\frac12}\|(1+\nu^{\frac13}t\mathcal{M})^{\theta}\pa_x \om_{2,j}\|_{L^1L^2},
\end{align*}
which together with \eqref{esti: f4}--\eqref{lanu tf3Y} implies
\begin{equation}
\begin{aligned}\label{esti: f3X}
\|f_3\|_{X_{\theta,j}}
\lesssim &\|h^{-1}_{\nu,t}f_1\|_{Y_j}+\nu^{\frac12}\|(1+\nu^{\frac13}t\mathcal{M})^{\theta}\pa_x \om_{2,j}\|_{L^1L^2}.
\end{aligned}
\end{equation}

To conclude the proof, it thus suffices to show
\begin{align}\label{kwreL2L1L2}
\nu^{\frac12}\|(1+\nu^{\frac13}t\mathcal{M})^{\theta}\pa_x \om_{2,j}\|_{L^1L^2} \lesssim 2^{-j}\nu^{-\frac13}\|(1+\nu^{\frac13}t \mathcal{M})^{\theta}u^{(2)}_e\|_{L^2(T_{j-1},T_{j};L^2)}E_0.
\end{align}
Let $w_{2,j}=\int_{\R}e^{-ikx}\om_{2,j} dx$. Due to $\min\{|k|^{\frac23},1\}\leq \mathcal{M}(k)$, we obtain
\begin{equation}\label{wreL2kL1tL2y}
\begin{aligned}
&\nu^{\frac12}\|(1+\nu^{\frac13}t\mathcal{M})^{\theta}\pa_x \om_{2,j}\|_{L^2_{x,y}}\\
\leq &\nu^{\frac16}\|\nu^{\frac13}(1+\nu^{\frac13}t\mathcal{M}(k) )^{\theta} (|k|^{\frac13}+ |k|)\min\{|k|^{\frac23},1\} w_{2,j}\|_{L^2_{k,y}} \\
\leq&\nu^{\frac16}\|\nu^{\frac13}(1+\nu^{\frac13}t\mathcal{M}(k) )^{\theta} (|k|^{\frac13}+ |k|)\mathcal{M}(k) w_{2,j}\|_{L^2_{k,y}} \\
\leq & C\nu^{\frac16}\| \nu^{\frac13}\mathcal{M} (1+\nu^{\frac13}t\mathcal{M})^{\theta} (1+\pa_x) \om_{2,j}\|_{ L^2_{x,y}}.
\end{aligned}
\end{equation}
Since $\om_{2,j}$ is the solution of \eqref{equ: om2j}, we use \eqref{esti: nu1/3M f} and Lemma \ref{lemma: fg esti} to deduce
\begin{equation*}
\begin{aligned}
&\nu^{\frac13}\|\mathcal{M} (1+\nu^{\frac13}t\mathcal{M})^{\theta} (1+\pa_x) \om_{2,j}\|_{L^1L^2}\\
\lesssim &\| (1+\nu^{\frac13}t\mathcal{M})^{\theta}(1+\pa_x)(tu^{(2)}_e\pa_x\om_{L}\chi_{j})\|_{L^1L^2}\\
\leq &\big\|t\|(1+\nu^{\frac13}t\mathcal{M})^{\theta}(1+\pa_x)\pa_x\om_{L}\|_{L^\infty}\|(1+\nu^{\frac13}t\mathcal{M})^{\theta}(1+\pa_x) u^{(2)}_e\|_{L^2}\big\|_{L^1}.
\end{aligned}
\end{equation*}
A direct computation gives
\begin{align*}
&\|(1+\nu^{\frac13}t\mathcal{M})^{\theta}(1+\pa_x)u^{(2)}_e\|_{L^2}= \int_{\R}(1+\nu^{\frac13}t\mathcal{M}(k))^{2\theta}(1+|k|)^2|k|^2\| \phi_e\|^2_{L^2_y}dk\\
\leq &2\int_{|k|\leq 1}(1+\nu^{\frac13}t\mathcal{M}(k))^{2\theta}|k|\| \pa_y\phi_e\|^2_{L^2_y}dk+2\int_{|k|\geq 1}(1+\nu^{\frac13}t\mathcal{M}(k))^{2\theta}|k|^2\| k\phi_e\|^2_{L^2_y}dk\\
\leq &2\|(1+\nu^{\frac13}t\mathcal{M})^{\theta}\mathcal{M}_1 \na\De^{-1}\om_e\|_{L^2_{x,y}}.
\end{align*}
We further use Proposition \ref{Pro: linear} to obtain
\begin{equation}\label{nu13omL1L2}
\begin{aligned}
&\nu^{\frac13}\|\mathcal{M} (1+\nu^{\frac13}t\mathcal{M})^{\theta} (1+\pa_x) \om_{2,j}\|_{L^1}\\
\lesssim& \big(\|t(1+\nu t^3)^{-1}\|_{L^2}+\|\nu^{-\frac13}(1+\nu t^3)^{-\frac76}\|_{L^2}\big)\|(1+\nu^{\frac13}t\mathcal{M})^{\theta}\mathcal{M}_1 \na\De^{-1}\om_e\|_{L^2(T_{j-1},T_j;L^2)} E_0\\
\lesssim & 2^{-j}\nu^{-\frac12}\|(1+\nu^{\frac13}t\mathcal{M})^{\theta}\mathcal{M}_1 \na\De^{-1}\om_e\|_{L^2(T_{j-1},T_j;L^2)} E_0.
\end{aligned}
\end{equation}
Inserting \eqref{nu13omL1L2} into \eqref{wreL2kL1tL2y}, we finally arrive at \eqref{kwreL2L1L2}.
\end{proof}
According to Lemmas \ref{lemma: f4 esti} and \ref{lemma: omre1}, we obtain
\begin{align*}
\|\om_{2,j}\|_{X_{\theta,j}}\leq \|f_3\|_{X_{\theta,j}}+\|f_4\|_{X_{\theta,j}}\lesssim \|h^{-1}_{\nu,t}f_1\|_{Y_j}+2^{-j}\nu^{-\frac13}\|\om_e\|_{X_{\theta,j}}E_0.
\end{align*}
Hence, the proof of \eqref{om2j large} is reduced to
$$\|h^{-1}_{\nu,t}f_1\|_{ Y_j}\lesssim 2^{-\frac{j}{4}}\nu^{-1/3}\|\om_e\|_{X_{\theta,j}}E_0,$$
 which is the content of the following lemma.
\begin{lemma}\label{lemma: f1Y+paxf1Y}
Let $f_1$ be the solution of \eqref{equ: f1}. Then, we have
\begin{align*}
\|h^{-1}_{\nu,t}f_1\|_{ Y_j}\leq C 2^{-\frac{j}{4}}\nu^{-1/3}\|\mathcal{M}_1 \na\De^{-1}\om_e\|_{L^2(T_{j-1},T_j;L^2)}E_0.
\end{align*}
\end{lemma}
\begin{proof}
Let $f_1(t,k,y)=\int_{\R} e^{-ikx} f_1(t,x,y)dx$.
In terms of Parseval's identity and Minkowski's  inequality, we obtain
\begin{align*}
\|h^{-1}_{\nu,t}f_1\|_{L^\infty_t L^2_{x,y} }=&\|h^{-1}_{\nu,t}f_1(t,k,y)\|_{L^\infty_t L^2_{k,y} }\leq \|f_1(t,k,y)\|_{L^2_k L^\infty_t L^2_y },\\
\|h^{-1}_{\nu,t}\mathcal{M}_1\na \De^{-1}f_1\|_{L^2_t L^2_{x,y} }\leq &\|(|k|^{\frac12}+ |k|)(\pa_y,k)(\pa^2_y-k^2)^{-1}f_1(t,k,y)\|_{L^2_k L^2_t L^2_y},
\end{align*}
which implies
\begin{align*}
\|h^{-1}_{\nu,t}f_1\|^2_{Y_j}\lesssim& \|f_1(t,k,y)\|^2_{L^2_k Y_{h,j}},
\end{align*}
where \begin{align*}
\|f_1(t,k,y)\|_{Y_{h,j}}=&\|f_1(t,k,y)\|_{L^\infty_t(T_{j-1},T_j; L^2_y)}+\nu^{\frac12}\|h^{-1}_{\nu,t}(\pa_y, k)f_1(t,k,y)\|_{L^2_t(T_{j-1},T_j; L^2_y)}\\
&+\|(|k|^{\frac12}+|k|)(\pa_y, k)(\pa^2_y-k^2)^{-1} f_1(t,k,y)\|_{L^2_t(T_{j-1},T_j; L^2_y)}.
\end{align*}
Let $B_{k}=\| (|k|^{\frac12}+|k|)(\pa_y, k)(\pa^2_y-k^2)^{-1}w_{e,k}\|_{L^2_t L^2_y}$.
To conclude the proof, it suffices to show
\begin{align*}
\|f_1(t,k,y)\|_{L^2_k Y_{h,j}}\leq C2^{-\frac{j}{4}}\nu^{-\frac13}\|B_{k}\|_{L^2_k}E_0.
\end{align*}
Recall $v^{(2)}_{e,k}(t,k,y)=\int_{\R} u^{(2)}_e e^{-ikx}dx$. Then, we have
\begin{align*}
[u^{(2)}_e\pa_x\om_{L}]^{\wedge}(t,k,y)=\int_{\R} ilv^{(2)}_{e,k-l} w^{L}_ldl.
\end{align*}
Let $f_{k,l}$ solve
\begin{align}\label{equ: fkl}
(\pa_t+iky)f_{k,l}(t,y)=ilth^2_{\nu,t} v^{(2)}_{e,k-l}  w^{L}_{l}=F_{k,l}(t,y),\quad f_{k,l}(T_0,y)=0.
\end{align}
It follows from \eqref{equ: f1} that
\begin{align*}
f_1(t,k,y)=\int_{\R}f_{k,l}(t,y)dl.
\end{align*}
As $\om_{L}|_{y=\pm 1}=0$, we have $w^L_{l}|_{y=\pm 1}=0$ and $F_{k,l}|_{y=\pm1}=0$.
Using Lemma \ref{lemma: om decom} for $s=\theta_1$, we get that
\begin{equation}\label{esti: fklYk}
\|f_{k,l}\|_{Y_{h,j}}
\leq 2^{(\frac32-\theta_1)j} (1+|k-l|+|l|)|k-l|^{-\frac12}\lan k-l\ran^{-\frac12}\|(\pa_y, \lan k-l\ran)(e^{il yt}F_{k,l})\|_{L^2_tL^2_y}.
\end{equation}
For $t\in(T_{j-1},T_j]$, we have
$$h^2_{\nu,t}t\leq \nu^{-\frac13}2^{(2\theta_1+1)j}.$$
Applying Lemma \ref{lemma:wL estimate}, we deduce
\begin{equation}\label{esti: wLl}
h^2_{\nu,t}t\|(\pa_y,1)(e^{il yt}l w^L_l)\|_{L^\infty_y}\lesssim \nu^{-\frac13}2^{(2\theta_1+1)j} \min\{|l|,|l|^{\frac12}\}A_2 e^{-2^{j}|l|^{2/3}}.
\end{equation}
In view of $F_{k,l}=ilth^2_{\nu,t}v^{(2)}_{e,k-l} w^{L}_l$, for $|k-l|\geq 1$, we have
\begin{equation}\label{pa Fk,l}
\begin{aligned}
\|(\pa_y,  k-l)(e^{il yt}F_{k,l})\|_{L^2_y}
\lesssim h^2_{\nu,t}t\|(\pa_y, k-l)v^{(2)}_{e,k-l}\|_{L^2_y}\|(\pa_y,1) e^{il yt}l w^L_l\|_{L^\infty_y}.
\end{aligned}
\end{equation}
Inserting \eqref{esti: wLl} into \eqref{pa Fk,l} and noticing that $\min\{|l|,|l|^{\frac12}\}A_2\leq \min\{|l|,|l|^{-\frac12}\}A_3$, we obtain
\begin{equation}\label{Fk,lk-l<1}
\begin{aligned}
&2^{(\frac32-\theta_1)j}\|(\pa_y, k-l)(e^{il yt} F_{k,l})\|_{L^2_tL^2_y}\\
\lesssim &2^{-\frac{j}{4}}\nu^{-\frac13}\|(\pa_y, k-l)v^{(2)}_{e,k-l}\|_{L^2_tL^2_y} |l|^{-\frac{11}{6}-\frac23\theta_1}\min\{|l|,|l|^{-\frac12}\}A_3\\
\lesssim &2^{-\frac{j}{4}}\nu^{-\frac13} B_{k-l}\min\{|l|^{-\frac56-\frac23\theta_1},|l|^{-\frac{7}{3}-\frac23\theta_1}\}A_3.
\end{aligned}
\end{equation}

For $|k-l|\leq 1$, due to
$$\|(\pa_y, 1)v^{(2)}_{e,k-l}\|_{L^2_y}=\|(k-l)(\pa_y, 1)(\pa^2_y-(k-l)^2)^{-1} w_{e,k}\|_{L^2_y}\leq |k-l|^{\frac12}B_{k-l}, $$
together with \eqref{esti: wLl}, we get
\begin{align}\label{Fk,lk-l>1}
2^{(\frac32-\theta_1)j}\|(\pa_y, 1)(e^{il yt} F_{k,l})\|_{L^2_tL^2_y}\lesssim &2^{-\frac{j}{4}}\nu^{-\frac13}|k-l|^{\frac12} B_{k-l}\min\{|l|^{-\frac56-\frac23\theta_1},|l|^{-\frac{7}{3}-\frac23\theta_1}\}A_3.
\end{align}
On the other hand, there holds
\begin{equation}\label{esti: kl}
\begin{aligned}
&(|k-l|+|l|)|k-l|^{-\frac12}\lan k-l\ran^{-\frac12}\lesssim 1+|l||k-l|^{-1},\, |k-l|\geq 1,\\
&(1+| k-l|+|l|)\lan k-l\ran^{-\frac12}\lesssim \lan l\ran,\quad |k-l|\leq 1.
\end{aligned}
\end{equation}
Inserting \eqref{Fk,lk-l<1}, \eqref{Fk,lk-l>1} and \eqref{esti: kl} into \eqref{esti: fklYk}, we deduce
\begin{align*}
\|f_1\|_{Y_{h,j}}\lesssim
&2^{-\frac{j}{4}}\int_{\R}\nu^{-\frac13}(|l|^{-\frac{4}{3}}A_3\widetilde{\chi}_{|l|\geq 1}+|l|^{-\frac56-\frac23\theta_1}A_3\widetilde{\chi}_{|l|\leq 1}) B_{k-l}dl.
\end{align*}
Integrating with $k$, we arrive at
\begin{equation}\label{f1L2Y}
\begin{aligned}
\|f_1\|_{L^2_kY_{h,j}}\lesssim & 2^{-\frac{j}{4}}\|B_{k}\|_{L^2_k}\Big(\int_{|l|\leq 1} \nu^{-\frac13}|l|^{-\frac56-\frac23\theta_1}A_3dl+\int_{|l|\geq 1} \nu^{-\frac13}|l|^{-\frac43}A_3dl\Big)\\
\lesssim &2^{-\frac{j}{4}}\nu^{-\frac13}\|B_{k}\|_{L^2_k}(\|A_3\|_{L^\infty_l}+\|A_3\|_{L^2_l})\lesssim 2^{-\frac{j}{4}}\nu^{-\frac13}\|B_{k}\|_{L^2_k}E_0,
\end{aligned}
\end{equation}
where  we used that $\theta_1\in (0, \frac14)$.
This completes the proof of Lemma \ref{lemma: f1Y+paxf1Y}.
\end{proof}
Now, we prove Proposition \ref{pro: wre}.
\begin{proof}
For \eqref{main result}, according to the decomposition of $\om_2=\sum_{j}\om_{2,j}$, we have
\begin{equation}
\begin{aligned}\label{om2XT0Tj0}
\|\om_2\|_{X(T_0,T_{j_0})}=&\Big\|\sum_{1\leq j\leq j_0} \om_{2,j}\Big\|_{X(T_0,T_{j_0})}\leq \sum_{1\leq j\leq j_0 }\|\om_{2,j}\|_{X(T_{j-1},T_{j_0})}\\
= &\sum_{1\leq j\leq j_0 }\|\om_{2,j}\|_{X(T_{j-1},T_j)}+\sum_{1\leq j\leq j_0-1}\|\om_{2,j}\|_{X(T_{j},T_{j_0})}.
\end{aligned}
\end{equation}
It follows from \eqref{om2j small} that
\begin{align}\label{om2jtildeXj}
\sum_{1\leq j\leq j_0 }\|\om_{2,j}\|_{X_j}\lesssim \nu^{-\frac13}\sum_{1\leq j\leq j_0}2^{-\frac{j}{4}}\|\mathcal{M}_1 \na\De^{-1}\om_e\|_{L^2(T_{j-1},T_j;L^2)} E_0\lesssim \nu^{-\frac13}\|\om_e\|_{X(T_0,T_{j_0})}E_0.
\end{align}
For $t>T_j$, the equation \eqref{equ: om2j} can be written as
\begin{equation}\label{equ: om2j free}
\left\{
\begin{aligned}
&\pa_t \om_{2,j}-\nu\De\om_{2,j}+y\pa_x\om_{2,j}=0, \quad t>T_{j}, \\
& \om_{2,j}(t,x,\pm1)=0, \quad \om_{2,j}|_{t=T_{j}}=\om_{2,j}(T_j).
\end{aligned}
\right.
\end{equation}
We use \eqref{esti: ftildeX} to obtain
\begin{align}\label{om2j Tj-Tj0}
\|\om_{2,j}\|_{X(T_{j},T_{j_0})}\leq C\|\om_{2,j}(T_j)\|_{L^2}\leq C\|\om_{2,j}\|_{X_j},
\end{align}
which together with \eqref{om2j small} yields that
\begin{align}\label{om2j tildeXj+1}
\|\om_{2,j}\|_{X(T_{j},T_{j_0})}\leq C2^{-\frac{j}{4}}\nu^{-\frac13}\|\mathcal{M}_1 \na\De^{-1}\om_e\|_{L^2(T_{j-1},T_j;L^2)} E_0.
\end{align}
Inserting \eqref{om2jtildeXj} and \eqref{om2j tildeXj+1} into \eqref{om2XT0Tj0}, we deduce
\begin{align}
\|\om_2\|_{X(T_0,T_{j_0})}\leq C\nu^{-\frac13}\|\om_e\|_{X(T_0,T_{j_0})}E_0.
\end{align}

Next, we go into the proof of \eqref{om2 Xtj}. Let
$$\om_2=\sum_{1\leq j\leq j_0}\om_{2,j}+\sum_{j\geq j_0+1}\om_{2,j}=:\widetilde{\om}_{2}+\sum_{j\geq j_0+1}\om_{2,j}.$$
For $\widetilde{\om}_{2}$, there holds
\begin{equation*}
\left\{
\begin{aligned}
&\pa_t \widetilde{\om}_{2}-\nu\De\widetilde{\om}_{2}+y\pa_x\widetilde{\om}_{2}=0, \quad t>T_{j_0}, \\
& \widetilde{\om}_{2}(t,x,\pm1)=0, \quad \widetilde{\om}_{2}|_{t=T_{j_0}}=\om_{2}(T_{j_0}).
\end{aligned}
\right.
\end{equation*}
Thanks to \eqref{fX large time}, we obtain
\begin{align}\label{om2,j small}
\|\widetilde{\om}_{2}\|_{X_{\theta}(T_{j_0},T)}\leq &C\|(1+\nu^{\frac13}T_{j_0} \mathcal{M})^{\theta} \widetilde{\om}_{2}(T_{j_0})\|_{L^2}=C\|(1+\nu^{\frac13}T_{j_0} \mathcal{M})^{\theta} \om_{2}(T_{j_0})\|_{L^2}.
\end{align}
For $j\geq j_0+1$, in a similar way to \eqref{om2jtildeXj}--\eqref{om2j tildeXj+1}, we use \eqref{om2j large} to obtain
\begin{align*}
\|\om_{2,j}\|_{X_{\theta}(T_{j-1},T)}\leq &\|\om_{2,j}\|_{X_{\theta}(T_{j-1},T_j)}+\|\om_{2,j}\|_{X_{\theta}(T_j,T)}\\
\leq &C2^{-\frac{j}{4}}\nu^{-\frac13}\|(1+\nu^{\frac13}t\mathcal{M})^{\theta}\mathcal{M}_1 \na\De^{-1}\om_e\|_{L^2(T_{j-1},T_j;L^2)}E_0,
\end{align*}
which implies
\begin{align}\label{om2,j large}
\sum_{j\geq j_0+1}\|\om_{2,j}\|_{X_{\theta}(T_{j},T)}\leq C\nu^{-\frac13}\|(1+\nu^{\frac13}t\mathcal{M})^{\theta}\mathcal{M}_1 \na\De^{-1}\om_e\|_{L^2(T_{j_0},T;L^2)} E_0.
\end{align}

Combining \eqref{om2,j large} and \eqref{om2,j small}, we directly obtain
\begin{align*}
\|\om_2\|_{X_{\theta}(T_{j_0},T)}\leq &\|\widetilde{\om}_{2}\|_{X_{\theta}(T_{j_0},T)}+\sum_{j\geq j_0+1}\|\om_{2,j}\|_{X_{\theta}(T_{j-1},T)}\\
\leq &C\nu^{-\frac13}\|\om_e\|_{X_{\theta}(T_{j_0}, T)}E_0+C\|(1+\nu^{\frac13}T_{j_0} \mathcal{M})^{\theta} \om_{2}(T_{j_0})\|_{L^2},
\end{align*}
which completes the proof of Proposition \ref{pro: wre}.
\end{proof}
 Based on these estimates, we prove Proposition \ref{ome estimate} via a bootstrap argument.
\begin{proof}
Thanks to \eqref{esti om1 tildeX} and \eqref{main result}, we obtain
\begin{align*}
\|\om_{e}\|_{X(T_0, T_{j_0})}\leq& \|\om_1\|_{X(T_0, T_{j_0})}+\|\om_2\|_{X(T_0, T_{j_0})}\\
\leq& C(\nu^{-\frac13}E_0\|\om_e\|_{X(T_0,T_{j_0})}+\nu^{-\frac58}\|\om_e\|^2_{X(T_0,T_{j_0})}+\nu^{\frac13-\epsilon}E_0),
\end{align*}
which implies
\begin{align*}
\|\om_{e}\|_{X(T_0, T_{j_0})}\leq C\nu^{-\frac23+\epsilon} \|\om_e\|^2_{X(T_0,T_{j_0})}+C\nu^{\frac13-\epsilon}E_0.
\end{align*}
Using the condition $E_0\leq c\nu^{\frac13}$, we deduce
\begin{align*}
\|\om_{e}\|_{X(T_0, T_{j_0})}\leq C\nu^{\frac13-\epsilon}E_0.
\end{align*}
Inserting it into \eqref{esti om1 tildeX} and \eqref{main result}, we also have
\begin{align*}
\|\om_{1}\|_{X(T_0, T_{j_0})}\leq C\nu^{\frac13-\epsilon}E_0, \quad  \|\om_{2}\|_{X(T_0, T_{j_0})}\leq C\nu^{\frac13-\epsilon}E_0.
\end{align*}
Choosing $0< \epsilon<(\frac{1}{24}-\frac23\theta)(1+\theta)^{-1}$ and $c_1> 8 C\nu^{\frac{1}{24}-\frac23\theta-\epsilon(1+\theta)}$, we obtain
\begin{equation}\label{ome small}
\begin{aligned}
&\|(1+ \nu^{\frac13}T_{j_0} \mathcal{M})^{\theta}\om_{1}(T_{j_0})\|_{L^2}\leq  \|(1+\nu^{\frac13}t\mathcal{M})^{\theta}\om_{1} \|_{X(T_0,T_{j_0})}\\
\leq & \|(1+\nu^{\frac13}t)^{\theta}\|_{L^\infty(T_0,T_{j_0})}\|\om_{1}\|_{X(T_0, T_{j_0})}
\leq  C \nu^{\frac13-\frac23\theta-\epsilon(1+\theta)}E_0\leq \frac18 c_1\nu^{\frac{7}{24}}E_0,
\end{aligned}
\end{equation}
and
\begin{equation}\label{om2 small}
\begin{aligned}
&\|(1+ \nu^{\frac13}T_{j_0} \mathcal{M})^{\theta}\om_{2}(T_{j_0})\|_{L^2}\leq  \|(1+\nu^{\frac13}t\mathcal{M})^{\theta}\om_{2} \|_{X(T_0,T_{j_0})}\\
\leq & \|(1+\nu^{\frac13}t)^{\theta}\|_{L^\infty(T_0,T_{j_0})}\|\om_{2}\|_{X(T_0, T_{j_0})}
\leq  C \nu^{\frac13-\frac23\theta-\epsilon(1+\theta)}E_0\leq \frac18 c_1\nu^{\frac{7}{24}}E_0.
\end{aligned}
\end{equation}
Assume that $T_{j_0}<T<\infty$ is the maximal time such that
\begin{align}\label{claim}
\|\om_e\|_{X_{\theta}(T_{j_0},T)}\leq c_1 \nu^{\frac{7}{24}} E_0.
\end{align}
Due to \eqref{esti om1 X large}, \eqref{om2 Xtj}, \eqref{ome small}, \eqref{om2 small} and the choice of $\epsilon$, we deduce that
\begin{equation}\label{omeX small}
\begin{aligned}
\|\om_e\|_{X_{\theta}(T_{j_0},T)}\leq& \|\om_{1}\|_{X_{\theta}(T_{j_0},T)}+\|\om_{2}\|_{X_{\theta}(T_{j_0},T)}\\
\leq &C\nu^{-\frac13}E_0\|\om_e\|_{X_{\theta}(T_{j_0},T)}+\nu^{-\frac58}\|\om_e\|^2_{X_{\theta}(T_{j_0},T)}+\nu^{\frac13}E_0+\frac14c_1 \nu^{\frac{7}{24}}E_0\\
\leq &C\nu^{-\frac58}\|\om_e\|^2_{X_{\theta}(T_{j_0},T)}+\frac13c_1 \nu^{\frac{7}{24}}E_0.
\end{aligned}
\end{equation}
Choosing $Ccc_1\leq \frac{1}{2}$, we then use \eqref{claim} to obtain
\begin{align*}
\|\om_{e}\|_{X_{\theta}(T_{j_0},T)}\leq Ccc_1 \|\om_{e}\|_{X_T}+\frac13c_1 \nu^{\frac{7}{24}}E_0
 \leq  \frac12\|\om_{e}\|_{X_{\theta}(T_{j_0},T)}+\frac13c_1 \nu^{\frac{7}{24}}E_0,
\end{align*}
which implies
\begin{align}\label{ome X}
\|\om_{e}\|_{X_{\theta}(T_{j_0},T)}\leq \frac23c_1 \nu^{\frac{7}{24}}E_0,
\end{align}
This is contradictory to the definition of $T$. Hence, we obtain $T=\infty$ and complete the proof of Proposition \ref{ome estimate}.
\end{proof}
\subsection{The proof of Theorems \ref{Th1}}\label{sec:proof th1}
According to the analysis in Section \ref{sec:Introduction}, we decompose the equation of $\om$ into $\om_{L}+\om_{e}$, whose estimates have already been established in Subsections \ref{sec:linear esti} and \ref{sec: error}, respectively.

Now, we present the proof of Theorem \ref{Th1}.
As for \eqref{Th: esti om}, it is clear that
\begin{align*}
\|(1+ \nu^{\frac13}t\mathcal{M})^{\theta}\om\|_{ L^2}\leq \|(1+ \nu^{\frac13}t\mathcal{M})^{\theta}\om_{L}\|_{ L^2}+\|(1+ \nu^{\frac13}t\mathcal{M})^{\theta}\om_{e}\|_{L^2}.
\end{align*}
Thanks to \eqref{paxpayomL L2} with $s=\theta$, we get
\begin{align*}
\|(1+ \nu^{\frac13}t\mathcal{M})^{\theta}\om_{L}\|_{ L^2}\leq CE_0.
\end{align*}
By Proposition \ref{Prop: esti ome}, we obtain
\begin{align*}
\|(1+ \nu^{\frac13}t\mathcal{M})^{\theta}\om_e\|_{ L^2}\leq CE_{0}.
\end{align*}
Combining the estimates of $\om_{L}$ and $\om_{e}$, we arrive at
\begin{align*}
\|(1+ \nu^{\frac13}t\mathcal{M})^{\theta}\om\|_{ L^2}\leq CE_0.
\end{align*}

For the inviscid damping estimate \eqref{inviscid}, we first divide $u=u_{L}+u_{e}$ as \eqref{equ: ome} to obtain
\begin{equation}
\begin{aligned}\label{esti u}
&\|(1+ \nu^{\frac13}t\mathcal{M})^{\theta}\mathcal{M}_1u\|_{L^2L^2}\\
\leq &\|(1+ \nu^{\frac13}   t\mathcal{M})^{\theta}\mathcal{M}_1u_{L}\|_{L^2L^2}+\|(1+ \nu^{\frac13} t\mathcal{M})^{\theta}\mathcal{M}_1u_{e}\|_{L^2L^2}.
\end{aligned}
\end{equation}
Using \eqref{k>1, psiL2}, \eqref{xi>1 xi4pay+xitpsiL2}, \eqref{esti4 psiL2}, \eqref{xi<1 pay+xitpsiL2}, \eqref{kneq0 wL2} and $\mathcal{M}(k)\leq 3|k|^{\frac23}$, we deduce
\begin{equation*}
\begin{aligned}
&\|(1+ \nu^{\frac13}t\mathcal{M})^{\theta}\mathcal{M}_1u_{L}\|^2_{L^2L^2}\\
 \leq &\int^{\infty}_0\int_{|k|\geq 1}(1+ \nu^{\frac13} t)^{2\theta}|k|^2\big((|k|+|k|t)\|\psi_{L}\|_{L^2_y}+\|(\pa_y+ikt)\psi_{L}\|_{L^2_y}\big)^2 dkdt\\
 &+3\int^{\infty}_0\int_{|k|\leq 1}(1+ \nu^{\frac13} |k|^{\frac23} t)^{2\theta}|k|((|k|+|k|t)\|\psi_{L}\|_{L^2_y}+\|(\pa_y+ikt)\psi_{L}\|_{L^2_y})^2 dkdt\\
\lesssim &\int^{\infty}_0 (1+t)^{-2} \int_{|k|\geq 1}  |A_3|^2 dkdt+\int^{\infty}_0  \int_{|k|\leq 1}\frac{|k|}{(1+|kt|)^2} |A_3|^2 dkdt.
\end{aligned}
\end{equation*}
Due to
\begin{align*}
\int_{|k|\leq 1}\frac{|k|}{(1+|kt|)^2}dk\leq (1+t)^{-1}\int_{|k|\leq 1}\frac{1}{1+|kt|}dk\leq C(1+t)^{-2}\ln(1+t),
\end{align*}
we obtain
\begin{equation}\label{uLneq}
\begin{aligned}
\|(1+\nu^{\frac13}t\mathcal{M})^{\theta}\mathcal{M}_1u_{L}\|^2_{L^2L^2}
\lesssim  \big(\|A_3\|^2_{L^2_k}+\|A_3\|^2_{L^\infty_k}\big)\int^{\infty}_0 (1+t)^{-\frac32}dt \lesssim E^2_0.
\end{aligned}
\end{equation}
Using Propositions \ref{small time} and \ref{ome estimate}, we have
\begin{equation}\label{ueneq}
\begin{aligned}
&\|(1+ \nu^{\frac13}t\mathcal{M})^{\theta}\mathcal{M}_1u_{e}\|^2_{L^2L^2}\\
=&\|(1+ \nu^{\frac13}t\mathcal{M})^{\theta}\mathcal{M}_1u_{e}\|^2_{L^2_t(0,\nu^{-\frac16})L^2_{x,y}}+\|(1+ \nu^{\frac13}t\mathcal{M})^{\theta}\mathcal{M}_1u_{e}\|^2_{L^2_t(\nu^{-\frac16},\infty)L^2_{x,y}}\\
\leq & \nu^{-\frac{1}{6}}\|(1+ \nu^{\frac13}t\mathcal{M})^{\theta}\om_e\|^2_{L^\infty_t(0,\nu^{-\frac16}) L^2_{x,y}}+ CE^2_0\lesssim E^2_0.
\end{aligned}
\end{equation}
Inserting \eqref{uLneq} and \eqref{ueneq} into \eqref{esti u}, we arrive at
\begin{align*}
\|(1+ \nu^{\frac13}t\mathcal{M})^{\theta}\mathcal{M}_1u\|_{L^2L^2}\leq C E_0.
\end{align*}

\appendix
\section{some basic estimates}
\begin{lemma}\label{operator estimate}
The operator $J_k$ is defined as
\begin{align*}
\mathfrak{J}_k[f](y)=| k|p.v. \frac{ k}{| k|}\int^1_{-1}\frac{1}{2i(y-y')}G_k(y,y')f(y')dy', \quad k\neq 0,
\end{align*}
with
\begin{equation*}
G_{k}(y,y')=-\frac{1}{ k\sinh 2 k}
\left\{
\begin{aligned}
&\sinh ( k(1-y'))\sinh( k(1+y)), \quad y\leq y',\\
&\sinh( k(1-y))\sinh( k(1+y')), \quad y\geq y'.
\end{aligned}
\right.
\end{equation*}
For $k\in \R\setminus \{0\}$, $\mathfrak{J}_k$ is bounded in $L^2$. That is, for $f\in L^2$, we have
\begin{align*}
\|\mathfrak{J}_k[f]\|_{L^2}\leq C\min\{1,|k|\}\|f\|_{L^2}.
\end{align*}
\end{lemma}
\begin{proof}
For $|k|\geq 1$, it follows from \cite[Lemma 7.1]{Arbon-Bedrossian} that
\begin{align}\label{JkfL2}
\|\mathfrak{J}_k[f]\|_{L^2}\leq C\|f\|_{L^2}.
\end{align}
For $k\leq1$, in a similar way to \cite{Arbon-Bedrossian,Bedrossian-He-Iyer-Wang}, noticing that
\begin{align*}
\|G_k(y,y)\|_{L^\infty}\leq C,\quad \sup_{y,y'}\frac{|G_k(y,y')-G_k(y,y)|}{|y-y'|}\leq C,
\end{align*}
we use the the boundedness  of the classical Hilbert transform in $L^2$ and Schur's test to obtain
\begin{align*}
\|J_k[f]\|_{L^2}\leq C|k|\|f\|_{L^2},\quad |k|\leq 1,
\end{align*}
which together with \eqref{JkfL2} completes the proof of Lemma \ref{operator estimate}.
\end{proof}
\begin{lemma}\label{lemma: fg esti}
Let $\mathcal{M}$ be the multiplier defined as \eqref{oper M}. For $s\geq 0$, there holds
\begin{align}\label{esti: fg}
\|(1+\nu^{\frac13}t\mathcal{M})^{s}(fg)\|_{L^2_{x,y}}\leq C \|(1+\nu^{\frac13}t\mathcal{M})^{s}f\|_{L^2_{x,y}} \|(1+\nu^{\frac13}t\mathcal{M})^{s}g\|_{L^\infty_{x,y}}.
\end{align}
\end{lemma}
\begin{proof}
According to the range of $k-l$ and $l$, we control $\mathcal{M}(k)$ by
\begin{align*}
\mathcal{M}(k)\leq \left\{
\begin{aligned}
& 1,\quad |l|\geq \frac14 \text{ or } |k-l|\geq \frac14,\\
& 3(1+|k-l|^{\frac23}+|l|^{\frac23}), \quad |l|\leq \frac14 \text{ and } |k-l|\leq \frac14,
\end{aligned}
\right.
\end{align*}
which yields
\begin{align*}
(1+\nu^{\frac13}t\mathcal{M}(k) )^s\leq &(1+\nu^{\frac13}t)^s\tilde{\chi}_{|k-l|\geq \frac14}+(1+\nu^{\frac13}t\mathcal{M}(k) )^s\tilde{\chi}_{|l|\geq \frac14}\\
&+3(1+\nu^{\frac13}t\mathcal{M}(k-l) )^s\chi_{|k-l|\leq \frac14}(1+\nu^{\frac13}t\mathcal{M}(l) )^s\chi_{|l|\leq \frac14}\\
\leq& C(1+\nu^{\frac13}t\mathcal{M}(k-l) )^s(1+\nu^{\frac13}t\mathcal{M}(l) )^s.
\end{align*}
Then, we use Plancherel's formula and H\"older's inequality to obtain \eqref{esti: fg}.
\end{proof}
Taking the Fourier transform in $x$-variable for \eqref{linear}, we obtain
\begin{equation}\label{linear Fourier}
\left\{
\begin{aligned}
&\pa_tw^{L}_k+\nu (1+t^2k^2)w^{L}_k+ikyw^{L}_k=0,\quad k\neq 0,\\
&w^{L}_k\big|_{t=0}=w^{in}(k,y),
\end{aligned}
\right.
\end{equation}
where $w^L_k= \int_{\R} e^{-ikx} \om_{L}dx$ and $w^{in}(k,y)=\int_{\R} e^{-ikx} \om^{in}dx$.
\begin{lemma}\label{lemma:wL estimate}
Let $w^{L}_k$ be the solution of \eqref{linear Fourier} with $t\geq 0$ and $A_j(k)=\|(\pa_y,k)^jw^{in}_{k}\|_{L^2_y}$.
Then,  we have
\begin{align}
&\|(\pa_y,k)^j (e^{ik yt} w^L_k)\|_{L^2_y}\leq Ce^{-\nu t-\frac13\nu k^2t^3}A_j(k), \label{kneq0 wL2}\\
&\|(\pa_y,k)^j (e^{ik yt} w^L_k)\|_{L^\infty_y}\leq C\min\{1,|k|^{-\frac12}\}e^{-\nu t-\frac13\nu k^2t^3}A_{j+1}(k). \label{kneq0 wLinfty}
\end{align}
\end{lemma}
A direct computation gives
\begin{equation*}
w^{L}_k=e^{-\nu t-\frac13\nu k^2 t^3-ikyt}w^{in}(k,y).
\end{equation*}
The estimates \eqref{kneq0 wL2}--\eqref{kneq0 wLinfty} can be obtained by the expression of $w^{L}_k$ and interpolation inequality.
\begin{lemma}
Assume that $\psi(k,y)=-(\pa^2_y-k^2)^{-1}w$ and $w(\pm1)=0$. There holds
\begin{align}
|k|^2\|(\pa_y,k)\psi\|_{L^2_y}\leq &C (1+t)^{-1}\|(\pa_y,k) (e^{ik yt} w)\|_{L^2_y}, \quad |k|\geq 1, \label{psiL2k>1}\\
|k|^2 \|(\pa_y,k)(\pa_y \psi+ik t \psi)\|_{L^2_y}\leq& C(1+t)^{-1}\|(\pa_y,k)^2 (e^{ik yt} w)\|_{L^2_y},\quad |k|\geq 1, \label{paypay+kpsiL2}\\
|k|^4\|\psi\|_{L^2_y}\leq& C(1+t)^{-2}\|(\pa_y,k)^2 (e^{ik yt} w)\|_{L^2_y},\quad |k|\geq 1,\label{k>1, psiL2}\\
| k|^4\|(\pa_y+i k t)\psi\|_{L^2_y}\leq & C(1+t)^{-2}\|(\pa_y,k)^3 (e^{ik yt} w)\|_{L^2_y}, \quad |k|\geq 1,\label{xi>1 xi4pay+xitpsiL2}
\end{align}
and
\begin{align}
\|(\pa_y,1)\psi\|_{L^2_y}\leq &C (1+|kt|)^{-1}\|(\pa_y,1) (e^{ik yt} w)\|_{L^2_y}, \quad |k|\leq 1, \label{psiL2k<1}\\
\|(\pa_y,1)(\pa_y\psi+ik t\psi)\|_{L^2_y}\leq& C(1+|k|t)^{-1}\|(\pa_y,1)^2 (e^{ik yt} w)\|_{L^2_y},\quad |k|\leq 1,\label{esti3 paypaypsi+xiytpsi}\\
\|\psi\|_{L^2_y}\leq& C(1+|kt|)^{-2}\|(\pa_y,1)^2 (e^{ik yt} w)\|_{L^2_y}, \quad |k|\leq 1,\label{esti4 psiL2}\\
\|(\pa_y+i k t)\psi\|_{L^2_y}\leq & C(1+| k t|)^{-2}\|(\pa_y,1)^3 (e^{ik yt} w)\|_{L^2_y}, \quad |k|\leq 1. \label{xi<1 pay+xitpsiL2}
\end{align}
\end{lemma}
\begin{proof}

We use the same method as in \cite[Lemma 3.4, Lemma 3.5]{Wei-Zhang-3} to derive \eqref{psiL2k>1}--\eqref{xi>1 xi4pay+xitpsiL2}. For the sake of completeness, we exhibit the proof of \eqref{psiL2k<1}--\eqref{xi<1 pay+xitpsiL2}, where $\|\cdot\|_{L^2}$ denotes $\|\cdot\|_{L^2_y(-1,1)}$.

Next, we go into the proof of \eqref{psiL2k<1}. Let $w_1=\pa_yw+ik tw$. As $\psi=-(\pa^2_y-k^2)^{-1} w$, we have
\begin{equation*}
\begin{aligned}
&|k t|(\|\pa_y \psi\|^2_{L^2}+|k|^2\|\psi\|^2_{L^2})=|k t||\lan \psi,w\ran|=|\lan \psi, w_1-\pa_y w\ran|\\
\leq &\|\psi\|_{L^2}\|w_1\|_{L^2}+\|\pa_y\psi\|_{L^2}\|w\|_{L^2}
\leq (\|w_1\|_{L^2}+\|w\|_{L^2})\|(\pa_y,1) \psi\|_{L^2}.
\end{aligned}
\end{equation*}
With the equality $e^{ik yt}w_1=\pa_y(e^{ik yt}w)$ and $\|\psi\|_{L^2}\leq \|\pa_y\psi\|_{L^2}$, we deduce
\begin{align}\label{kt pay,1psi L2}
\frac12|k t|\|(\pa_y,1) \psi\|_{L^2}\leq \|(\pa_y,1) (e^{ik yt}w)\|_{L^2}.
\end{align}
Moreover, due to $\|\pa_y \psi\|^2_{L^2}\leq\lan \psi,w\ran\leq \|\psi\|_{L^2}\|w\|_{L^2}$, we obtain
\begin{align}\label{pay,1psi L2}
\frac12\|(\pa_y,1) \psi\|_{L^2}\leq  \|e^{ik yt}w\|_{L^2}.
\end{align}
The estimates \eqref{kt pay,1psi L2} and \eqref{pay,1psi L2} imply \eqref{psiL2k<1}.

Define
$$ \psi_1=\pa_y \psi+ik t\psi, \quad \psi_2=-(\pa^2_y-k^2)^{-1}w_1,\quad \psi_3=\psi_1-\psi_2.$$
Due to $-(\pa^2_y-k^2)\psi=w$, we have
$$ -(\pa^2_y-k^2)\psi_1=w_1, \quad (\pa^2_y-k^2) \psi_3=0.$$
Thanks to the definition of $\psi_1$, there holds
\begin{align}\label{esti: pay+ixiytpsiL2}
\|(\pa_y,1)(\pa_y\psi+ik t\psi)\|_{L^2}=\|(\pa_y,1)\psi_1\|_{L^2}\leq \|(\pa_y,1)\psi_2\|_{L^2}+\|(\pa_y,1)\psi_3\|_{L^2}.
\end{align}

By the definition of $\psi_2$, $e^{ikyt}w_1=\pa_y(e^{ikty} w)$, we use \eqref{psiL2k<1} to obtain
\begin{equation}\label{esti: paypsi2L2}
\begin{aligned}
 \|(\pa_y,1)\psi_2\|_{L^2}\lesssim &  (1+|k t|)^{-1}\|(\pa_y,1)(e^{ik yt} w_1)\|_{L^2}\\
 =&(1+|k t|)^{-1}\|(\pa_y,1)\pa_y( e^{ik yt} w)\|_{L^2}.
\end{aligned}
\end{equation}

From the equation of $\psi_3$, we have
\begin{align*}
\psi_3=\psi_3(-1)\gamma_{-1}+\psi_3(1)\gamma_{1},
\end{align*}
where
$\gamma_j(y)=\frac{\sinh (k(y+j))}{j\sinh 2k}$ is the solution of $(\pa^2_y-k^2)\gamma_j=0,$ $\gamma_j(j)=1$, $\gamma_j(-j)=0$ with $j=\pm1$.
Thanks to $|\gamma'_j(y)|=|\frac{k\cosh(k(y+j))}{j\sinh 2k}|\leq C$ for $|k|\leq 1$, $y\in[-1,1]$ and $j\in\{\pm 1\}$, we arrive at
\begin{equation}\label{gammaL2}
\begin{aligned}
\frac12\|(\pa_y, 1) \gamma_j\|^2_{L^2}\leq& \|\pa_y \gamma_j\|^2_{L^2}+|k|^2\|\gamma_j\|^2_{L^2}\\
=&-\lan \gamma_j, (\pa^2_y-k^2)\gamma_j\ran+\gamma'_j \gamma_j|^{1}_{-1}=|\gamma'_j \gamma_j(j)|\leq C,
\end{aligned}
\end{equation}
which implies
\begin{equation}\label{esti: paypsi3L2 mid}
\begin{aligned}
\|(\pa_y, 1)\psi_3\|_{L^2}\leq  &|\psi_3(-1)|\|(\pa_y, 1)\gamma_{-1}\|_{L^2}+|\psi_3(1)|\|(\pa_y, 1)\gamma_{1}\|_{L^2}\\
\lesssim &(|\pa_y\psi(-1)|+|\pa_y\psi(1)|),
\end{aligned}
\end{equation}
where we used that $\psi_3(\pm 1)=\psi_1(\pm1)=\pa_y\psi(\pm 1)$.
Using the fact that
\begin{align}\label{pay psi(1)}
-\lan w^L_k,\gamma_1\ran=\lan (\pa^2_y-k^2)\psi, \gamma_1\ran=\lan \psi,(\pa^2_y-k^2)\gamma_1\ran+(\pa_y \psi\gamma_1-\psi \gamma'_1)|^{1}_{-1}=\pa_y \psi(1),
\end{align}
 and $e^{ik yt}w_1=\pa_y(e^{ik yt}w)$, \eqref{gammaL2}, we infer that
\begin{equation}\label{kt pay psi(1)}
\begin{aligned}
|k t\pa_y \psi(1)|=&|k t\lan w,\gamma_1\ran|=|\lan w_1-\pa_yw,\gamma_1\ran|= \big|\lan w_1,\gamma_1\ran+\lan w,\gamma'_1\ran-w\gamma_1|^{1}_{-1}\big|\\
\leq&\|w_1\|_{L^2}\|\gamma_1\|_{L^2}+\|w\|_{L^2}\|\gamma'_1\|_{L^2}\\
\leq&C(\|w_1\|_{L^2}+\|w\|_{L^2})\leq C\|(\pa_y,1)(e^{ik yt}w)\|_{L^2}.
\end{aligned}
\end{equation}
On the other hand, by \eqref{pay psi(1)}, we obtain
\begin{align}\label{paypsicontrol byomL2}
|\pa_y \psi(1)|=|\lan w,\gamma_1\ran|\leq \|w\|_{L^2}\|\gamma_1\|_{L^2}\leq \|e^{ikty}w\|_{L^2}.
\end{align}
The estimate \eqref{kt pay psi(1)} together with \eqref{paypsicontrol byomL2} implies
\begin{align}\label{paypsi1}
|\pa_y \psi(1)|\leq C(1+|k|t)^{-1}\|(\pa_y,1)(e^{ik yt}w)\|_{L^2}.
\end{align}
Similarly, we have
\begin{align}\label{paypsi-1}
|\pa_y\psi(-1)|\leq C(1+|k|t)^{-1}\|(\pa_y,1)(e^{ik yt}w)\|_{L^2}.
\end{align}
Inserting \eqref{paypsi1} and \eqref{paypsi-1} into \eqref{esti: paypsi3L2 mid}, we arrive at
\begin{align}\label{esti: paypsi3L2}
\|(\pa_y, 1)\psi_3\|_{L^2}\leq C(1+|k|t)^{-1}\|(\pa_y,1)(e^{ik yt}w)\|_{L^2}.
\end{align}
It follows from \eqref{esti: paypsi2L2} and \eqref{esti: paypsi3L2} that
\begin{equation}
\begin{aligned}\label{psi1L2}
\|(\pa_y, 1)\psi_1\|_{L^2}\leq &C \|(\pa_y, 1)\psi_2\|_{L^2}+\|(\pa_y, 1)\psi_3\|_{L^2}\\
\leq &C(1+|k|t)^{-1}\|(\pa_y,1)\pa_y(e^{ik yt}w)\|_{L^2}.
\end{aligned}
\end{equation}

For \eqref{esti4 psiL2}, by $\psi_1=\pa_y \psi+ik t\psi$, we have
\begin{align}\label{xitpsiL2}
| k t|\|\psi\|_{L^2}\leq \|\pa_y \psi\|_{L^2}+\|\psi_1\|_{L^2}.
\end{align}
Using \eqref{psiL2k<1} and \eqref{psi1L2}, we get
\begin{equation}\label{esti: paypsi+psi1L2}
\begin{aligned}
\|\pa_y \psi\|_{L^2}+\|\psi_1\|_{L^2}\leq &C(1+| k|t)^{-1}(\|(\pa_y,1)(e^{i k yt} w)\|_{L^2}+\|(\pa_y,1)^2(e^{i k yt} w)\|_{L^2})\\
\leq& C(1+| k|t)^{-1}\|(\pa_y,1)^2(e^{i k yt} w)\|_{L^2}.
\end{aligned}
\end{equation}
Inserting \eqref{esti: paypsi+psi1L2} into \eqref{xitpsiL2} and then combining it with \eqref{psiL2k<1}, we arrive at
\begin{align*}
\|\psi\|_{L^2}\leq C(1+| k|t)^{-2}\|(\pa_y,1)^2(e^{i k yt} w)\|_{L^2}.
\end{align*}

For \eqref{xi<1 pay+xitpsiL2}, due to
\begin{align*}
\|(\pa_y+i k t)\psi\|_{L^2}=\|\psi_1\|_{L^2}\leq \|\psi_2\|_{L^2}+\|\psi_3\|_{L^2},
\end{align*}
we are left to bound $\psi_2$ and $\psi_3$.
We use \eqref{esti4 psiL2} and $e^{ikyt}w_1=\pa_y(e^{ikty} w)$ to get
\begin{align}\label{esti: psi2}
\|\psi_2\|_{L^2}\leq C(1+| k|t)^{-2}\|(\pa_y,1)^2(e^{i k yt}w_1)\|_{L^2}
\leq C(1+| k|t)^{-2}\|(\pa_y,1)^3(e^{i k yt}w)\|_{L^2}.
\end{align}
For $\psi_3$, thanks to $w=-(\pa^2_y-k^2)\psi$, we have
\begin{equation}\label{esti: kt2paypsi1}
\begin{aligned}
| k t|^2|\pa_y\psi(1)|=&| k t|^2|\lan w,\gamma_1\ran|=| k t|^2|\lan e^{i k yt}w \gamma_1, e^{i k yt}\ran|=|\lan e^{i k yt}w\gamma_1, \pa^2_y(e^{i k yt})\ran|\\
=&\left|\lan \pa^2_y(e^{i k yt}w\gamma_1),e^{i k yt}\ran-e^{i k yt}\pa_y(e^{i k yt}w\gamma_1)|^{1}_{-1}\right|\\
\leq &\|\pa^2_y(e^{i k yt}w\gamma_1)\|_{L^1}+\|\pa_y(e^{i k yt}w \gamma_1)\|_{L^\infty}.
\end{aligned}
\end{equation}
For any $y\in[-1,1]$, thanks to  $\pa_y(e^{i k yt} w\gamma_1)|_{y=-1}=0$, we have
\begin{align*}
|\pa_y(e^{i k yt}w\gamma_1)(y)|= \Big|\int^{y}_{-1}\pa^2_{y}(e^{i k yt}w \gamma_1) dy\Big|\leq \|\pa^2_{y}(e^{i k yt}w \gamma_1)\|_{L^1}.
\end{align*}
Inserting it into \eqref{esti: kt2paypsi1}, we obtain
\begin{align*}
| k t|^2|\pa_y\psi(1)|
\lesssim \|\pa^2_y(e^{i k yt}w\gamma_1)\|_{L^1}
\lesssim \sum_{j=0}^{2}\|\pa^j_y(e^{i k yt}w)\|_{L^2}\|\pa^{2-j}_y\gamma_1\|_{L^2}\lesssim \|(\pa_y,1)^2(e^{i k yt}w)\|_{L^2}.
\end{align*}
This together with \eqref{paypsicontrol byomL2} implies
\begin{align}\label{paypsi(1) 2}
|\pa_y\psi(1)|\leq C(1+| k|t)^{-2}\|(\pa_y,1)^2(e^{i k yt}w)\|_{L^2}.
\end{align}
Similarly, we derive that
\begin{align}\label{paypsi(-1) 2}
|\pa_y\psi(-1)|\leq C(1+| k|t)^{-2}\|(\pa_y,1)^2(e^{i k yt}w)\|_{L^2}.
\end{align}
Inserting \eqref{paypsi(1) 2} and \eqref{paypsi(-1) 2} into \eqref{esti: paypsi3L2 mid}, we arrive at
\begin{align*}
\|\psi_3\|_{L^2}\leq C(1+| k|t)^{-2}\|(\pa_y,1)^2(e^{i k yt}w)\|_{L^2}.
\end{align*}
This inequality together with \eqref{esti: psi2} yields
\begin{align*}
\|(\pa_y \psi+i k yt\psi)\|_{L^2}\leq C(1+| k|t)^{-2}\|(\pa_y,1)^2(e^{i k yt}w)\|_{L^2}.
\end{align*}
\end{proof}

\section*{Declarations}
\begin{itemize}
  \item \textbf{Acknowledgements} The authors thank to Prof. Dongyi Wei for many valuable discussions. This project was supported by the National Key Research and Development Program of China (No: 2022YFA1005700). Q. Chen was partially supported by National Natural Science Foundation of China (No: 12471149). Z. Li was partially supported by the Postdoctoral Fellowship Program of CPSF (No: GZC20240123). C. Miao was partially supported by National Natural Science Foundation of China (No: 12371095).
  \item \textbf{Conflict of interest} The authors declare that they have no conflict of interest.
  \item \textbf{Data Availability} Data sharing is not applicable to this article as no datasets were generated or analyzed during
the current study.

\end{itemize}

\end{document}